\documentclass[a4paper,reqno,10pt]{amsart}

\usepackage{hyperref}
\usepackage{a4wide}

\usepackage{amssymb}
\usepackage{amstext}
\usepackage{amsmath}
\usepackage{amscd}
\usepackage{amsthm}
\usepackage{amsfonts}

\usepackage{graphicx}
\usepackage{latexsym}
\usepackage{mathrsfs}
\usepackage{xy}
\usepackage{enumitem}
\setlist[enumerate,1]{label={\upshape(\arabic*)}}
\setlist[enumerate,2]{label={\upshape(\alph*)}}
\usepackage{wasysym} 
\usepackage{tikz}
\usetikzlibrary{cd,arrows,matrix,backgrounds,positioning,calc,decorations.markings,decorations.pathmorphing,decorations.pathreplacing}
\tikzset{black/.style={circle,fill=black,inner sep=3pt,outer sep=3pt},
	white/.style={circle,fill=white,draw=black,inner sep=3pt,outer sep=3pt},
}

\usepackage{array}
\newcolumntype{C}{>{$}c<{$}}

\usepackage{makecell}
\usepackage{adjustbox}

\newtheorem{theorem}{Theorem}[section]
\newtheorem{theoremi}{Theorem}
\newtheorem{corollaryi}[theoremi]{Corollary}

\renewcommand{\thetheoremi}{\Alph{theoremi}}
\newtheorem{corollary}[theorem]{Corollary}
\newtheorem{lemma}[theorem]{Lemma}
\newtheorem*{lemma*}{Lemma}
\newtheorem*{theorem*}{Theorem}
\newtheorem{proposition}[theorem]{Proposition}
\newtheorem{definition-proposition}[theorem]{Definition-Proposition}

\theoremstyle{definition}
\newtheorem{definition}[theorem]{Definition}

\newtheorem{remark}[theorem]{Remark}
\newtheorem{example}[theorem]{Example}

\xyoption{all}

\newcommand{\qedb}{\hfill\blacksquare}

\newcommand{\la}{\langle}
\newcommand{\ra}{\rangle}
\newcommand{\al}{\alpha}
\newcommand{\be}{\beta}
\newcommand{\ga}{\gamma}
\newcommand{\vare}{\varepsilon}

\newcommand{\mm}{{\mathfrak{m}}}
\newcommand{\nn}{{\mathfrak{n}}}
\newcommand{\pp}{{\mathfrak{p}}}
\newcommand{\qq}{{\mathfrak{q}}}
\renewcommand{\AA}{\mathcal{A}}
\newcommand{\AAA}{\mathfrak{A}}
\newcommand{\BBB}{\mathfrak{B}}
\newcommand{\CC}{\mathcal{C}}
\newcommand{\MM}{\mathcal{M}}
\newcommand{\NN}{\mathcal{N}}
\newcommand{\GG}{\mathcal{G}}
\newcommand{\CCC}{\mathsf{C}}
\newcommand{\DD}{\mathcal{D}}
\newcommand{\DDD}{\mathsf{D}}
\newcommand{\KK}{\mathsf{K}}
\newcommand{\FF}{\mathcal{F}}
\newcommand{\FFF}{\mathsf{F}}
\newcommand{\GGG}{\mathsf{G}}
\newcommand{\HH}{\mathcal{H}}
\newcommand{\PP}{\mathsf{P}}
\newcommand{\II}{\mathcal{I}}
\newcommand{\JJ}{\mathcal{J}}
\newcommand{\OO}{\mathcal{O}}
\newcommand{\RR}{\mathcal{R}}
\renewcommand{\SS}{\mathcal{S}}
\newcommand{\TT}{\mathcal{T}}
\newcommand{\UU}{\mathsf{U}}
\newcommand{\WW}{\mathcal{W}}
\newcommand{\WWW}{\mathsf{W}}
\newcommand{\XX}{\mathcal{X}}
\newcommand{\XXX}{\mathsf{X}}
\newcommand{\YY}{\mathcal{Y}}

\renewcommand{\L}{\mathbb{L}}
\newcommand{\Z}{\mathbb{Z}}
\renewcommand{\P}{\mathbb{P}}
\newcommand{\I}{\mathcal{I}}
\newcommand{\Q}{\mathbb{Q}}
\newcommand{\R}{\mathbb{R}}
\newcommand{\N}{\mathbb{N}}
\newcommand{\X}{\mathbb{X}}
\newcommand{\bo}{\operatorname{b}\nolimits}
\newcommand{\acy}{\operatorname{acy}\nolimits}
\newcommand{\cpt}{\operatorname{cpt}\nolimits}
\newcommand{\depth}{\operatorname{depth}\nolimits}
\newcommand{\length}{\operatorname{length}\nolimits}
\renewcommand{\top}{\operatorname{top}\nolimits}
\newcommand{\cotop}{\operatorname{cotop}\nolimits}
\newcommand{\Bcotop}{B\operatorname{-cotop}\nolimits}
\newcommand{\Bcorad}{B\operatorname{-corad}\nolimits}
\newcommand{\soc}{\operatorname{soc}\nolimits}
\newcommand{\rad}{\operatorname{rad}\nolimits}
\newcommand{\corad}{\operatorname{corad}\nolimits}
\newcommand{\Ext}{\operatorname{Ext}\nolimits}
\newcommand{\EExt}{\operatorname{\mathcal{E}}\nolimits}
\newcommand{\Tor}{\operatorname{Tor}\nolimits}
\newcommand{\Hom}{\operatorname{Hom}\nolimits}
\newcommand{\Rad}{\operatorname{Rad}\nolimits}
\newcommand{\End}{\operatorname{End}\nolimits}
\newcommand{\Aut}{\operatorname{Aut}\nolimits}
\newcommand{\gl}{\operatorname{gl.\!dim}\nolimits}
\newcommand{\height}{\operatorname{ht}\nolimits}
\newcommand{\dm}{\operatorname{dom.\!dim}\nolimits}
\newcommand{\op}{\operatorname{op}\nolimits}
\newcommand{\RHom}{\mathbf{R}\strut\kern-.2em\operatorname{Hom}\nolimits}
\newcommand{\Lotimes}{\mathop{\stackrel{\mathbf{L}}{\otimes}}\nolimits}
\newcommand{\Image}{\operatorname{Im}\nolimits}
\newcommand{\Kernel}{\operatorname{Ker}\nolimits}
\newcommand{\Cokernel}{\operatorname{Coker}\nolimits}
\newcommand{\Spec}{\operatorname{Spec}\nolimits}
\newcommand{\Supp}{\operatorname{Supp}\nolimits}
\newcommand{\MaxSpec}{\operatorname{MaxSpec}\nolimits}
\newcommand{\Ass}{\operatorname{Ass}\nolimits}
\newcommand{\ann}{\operatorname{ann}\nolimits}

\newcommand{\rank}{\operatorname{rank}}

\newcommand{\MaxProj}{\operatorname{MaxProj}\nolimits}
\newcommand{\sign}{\operatorname{sgn}\nolimits}
\newcommand{\GL}{\operatorname{GL}\nolimits}
\newcommand{\SL}{\operatorname{SL}\nolimits}
\newcommand{\E}{\operatorname{E}\nolimits}
\newcommand{\Ab}{\mathcal{A}b}
\newcommand{\coker}{\Cokernel}
\newcommand{\im}{\Image}
\renewcommand{\ker}{\Kernel}
\newcommand{\un}{\underline}
\newcommand{\ov}{\overline}

\newcommand{\cov}{\gtrdot}
\newcommand{\covd}{\lessdot}

\newcommand{\can}{\mathsf{can}}
\newcommand{\red}{\mathsf{red}}

\newcommand{\Aff}{\operatorname{Aff}\nolimits}

\newcommand{\Ho}{{\rm H}}
\newcommand{\LHo}{\underline{\rm H}}

\newcommand{\ot}{\leftarrow}

\DeclareMathOperator{\moduleCategory}{\mathsf{mod}} \renewcommand{\mod}{\moduleCategory}
\DeclareMathOperator{\Mod}{\mathsf{Mod}}
\DeclareMathOperator{\Ob}{\mathsf{Ob}}
\DeclareMathOperator{\proj}{\mathsf{proj}}
\DeclareMathOperator{\Proj}{\mathsf{Proj}}
\DeclareMathOperator{\inj}{\mathsf{inj}}
\DeclareMathOperator{\ind}{\mathsf{ind}}
\DeclareMathOperator{\simp}{\mathsf{sim}}
\DeclareMathOperator{\indp}{\mathsf{ind}\mathcal{P}}
\DeclareMathOperator{\indi}{\mathsf{ind}\mathcal{I}}

\DeclareMathOperator{\cont}{\mathsf{cont}}

\DeclareMathOperator{\torf}{\mathsf{torf}}
\DeclareMathOperator{\Hasse}{\mathsf{Hasse}}
\DeclareMathOperator{\brick}{\mathsf{brick}}

\DeclareMathOperator{\sbrick}{\mathsf{sbrick}}
\DeclareMathOperator{\mbrick}{\mathsf{mbrick}}

\DeclareMathOperator{\ccmbrick}{\mathsf{mbrick_{c.c.}}}

\DeclareMathOperator{\lSchur}{\mathsf{Schur_L}}
\DeclareMathOperator{\Serre}{\mathsf{Serre}}

\DeclareMathOperator{\wide}{\mathsf{wide}}
\DeclareMathOperator{\jirr}{\mathsf{j-irr}}
\DeclareMathOperator{\ftorf}{\mathsf{f-torf}}

\DeclareMathOperator{\mmax}{\mathsf{max}}

\DeclareMathOperator{\gp}{\mathsf{gp}}

\DeclareMathOperator{\Iso}{\mathsf{Iso}}
\DeclareMathOperator{\Atom}{\mathsf{Atom}}

\DeclareMathOperator{\Sub}{\mathsf{Sub}}
\DeclareMathOperator{\sub}{\mathsf{sub}}

\DeclareMathOperator{\Fac}{\mathsf{Fac}}
\DeclareMathOperator{\Filt}{\mathsf{Filt}}
\DeclareMathOperator{\thick}{\mathsf{thick}}
\DeclareMathOperator{\Loc}{\mathsf{Loc}}
\DeclareMathOperator{\coh}{\mathsf{coh}}
\DeclareMathOperator{\per}{\mathsf{per}}
\DeclareMathOperator{\Ex}{\mathsf{Ex}}
\DeclareMathOperator{\AR}{\mathsf{AR}}
\DeclareMathOperator{\grproj}{\mathsf{gr-proj}}
\DeclareMathOperator{\Gr}{\mathsf{Gr}}
\DeclareMathOperator{\gr}{\mathsf{gr}}
\DeclareMathOperator{\qGr}{\mathsf{qGr}}
\DeclareMathOperator{\qgr}{\mathsf{qgr}}
\DeclareMathOperator{\CM}{\mathsf{CM}}
\DeclareMathOperator{\GP}{\mathsf{GP}}
\DeclareMathOperator{\LCM}{\mathsf{LCM}}
\DeclareMathOperator{\Sing}{\mathsf{Sing}}
\DeclareMathOperator{\add}{\mathsf{add}}
\DeclareMathOperator{\fl}{\mathsf{fl}}
\DeclareMathOperator{\fd}{\mathsf{fd}}
\DeclareMathOperator{\pd}{\mathsf{pd}}
\DeclareMathOperator{\id}{\mathsf{id}}
\DeclareMathOperator{\noeth}{\mathsf{noeth}}
\DeclareMathOperator{\art}{\mathsf{art}}

\DeclareMathOperator{\inv}{\mathsf{inv}}
\DeclareMathOperator{\Binv}{\mathsf{Binv}}
\DeclareMathOperator{\nsp}{\mathsf{nsp}}

\DeclareMathOperator{\Inv}{\mathsf{Inv}}
\DeclareMathOperator{\BInv}{\mathsf{BInv}}

\DeclareMathOperator{\supp}{\mathsf{supp}}
\DeclareMathOperator{\udim}{\un{\dim}}
\DeclareMathOperator{\cone}{\mathsf{cone}}
\DeclareMathOperator{\cocone}{\mathsf{cocone}}
\def\lr#1{\langle #1\rangle}

\newcommand{\iso}{\cong}
\newcommand{\infl}{\rightarrowtail}
\newcommand{\defl}{\twoheadrightarrow}

\newcommand{\equi}{\simeq}

\newcommand{\hc}[1]{\left[ #1 \right)}

\newcommand{\A}{\mathbb{A}}
\newcommand{\D}{\mathbb{D}}
\newcommand{\Cf}{\mathbb{C}}
\newcommand{\BB}{\mathcal{B}}
\newcommand{\EE}{\mathcal{E}}

\numberwithin{equation}{section}

\renewcommand{\baselinestretch}{1}

\begin{document}
	\title[The Grothendieck group of an extriangulated category]{The Grothendieck group of an extriangulated category}
	
	\author[L. Wang]{Li Wang}
	\address{School of Mathematics-Physics and Finance, Anhui Polytechnic University, 241000 Wuhu, Anhui, P. R. China}
	\email{wl04221995@163.com {\rm (L. Wang)}}

	\begin{abstract}
		In this paper, we investigate the split Grothendieck group $K^{\rm sp}_{0}(\mathcal{M})$ of a $d$-rigid subcategory $\mathcal{M}$ in an extriangulated category $\mathscr{C}$. As applications, we prove the following  results: (1) If $\mathcal{M}$ is a silting subcategory, then the Grothendieck group $K_{0}(\mathscr{C})$ is isomorphic to   $K_{0}^{\rm sp}(\mathcal{M})$;  (2) If $\mathcal{M}$ is a $d$-cluster tilting subcategory, then  $K_{0}(\mathscr{C})$ is isomorphic to the index Grothendieck group $K_{0}^{\rm in}(\mathcal{M})$; (3) Let $\mathcal{C}_{A_{n}}^{d}$ be the $d$-cluster category of type $A_n$. If $d$ is even, then  $K_0(\mathcal{C}_{A_{n}}^{d})\cong \mathbb{Z}/(n+1)\mathbb{Z}$. If $d$ is odd, then $K_0(\mathcal{C}_{A_{n}}^{d})\cong \mathbb{Z}$ if $n$ is odd; $K_0(\mathcal{C}_{A_{n}}^{d})\cong 0$ if $n$ is even.		
	\end{abstract}
	\maketitle

	
	\section{Introduction}
	The Grothendieck group of a triangulated category is defined as the free abelian group generated by the isomorphism classes of its objects modulo the Euler relations formed by  triangles. 	In a similar manner, the Grothendieck group of an exact category can be defined using conflations instead of triangles. It is known that these groups are closely related to several important topics, e.g.  classifications \cite{Tho}, silting theory \cite{Ai}, cluster theory \cite{FE} and so on.
	
	Recently, Nakaoka and Palu \cite{Na} introduced the notion of extriangulated categories by extracting properties on triangulated categories and exact categories. It turns out that many classical concepts and theories concerning exact and triangulated categories have been unified and extended in this setting, e.g. \cite{Go,He,Wl}. The notion of Grothendieck groups of  extriangulated categories was first introduced in \cite{Zhu} to study the Auslander-Reiten theory and the  classification of dense resolving subcategories. 	It seems to be an interesting but difficult problem to compute the Grothendieck  group $K_{0}(\mathscr{C})$  of a given extriangulated category $\mathscr{C}$. This might be done by connecting it to a split Grothendieck group  $K^{\rm sp}_{0}(\mathcal{M})$ of a specific subcategory $\mathcal{M}$. We present some related results:
	
	(1) The first one is based on the silting subcategories in silting theory. Let $\mathscr{C}$ be a Krull-Schmidt triangulated category with a silting subcategory $\mathcal{M}$. It is well-known that the inclusion $\mathcal{M}\hookrightarrow\mathscr{C}$ induces an isomorphism between  $K_{0}(\mathscr{C})$ and  $K^{\rm sp}_{0}(\mathcal{M})$.  This result was proved  by Aihara and Iyama in \cite[Theorem 2.27]{Ai} and generalized by  Adachi and Tsukamoto in \cite[Theorem 4.1]{Ad2} for  Krull-Schmidt extriangulated categories. In \cite[Theorem 3.3]{Ch}, the authors showed that this result also applies to any triangulated categories. 
	
	$(2)$ The second one arises from the  cluster tilting subcategories in cluster theory. Let $\mathscr{C}$ be a triangulated category of finite type.  Xiao and Zhu \cite{Xi} proved that  $K_{0}(\mathscr{C})$ is  isomorphic to  $K_{0}^{\rm sp}(\mathscr{C})$ modulo  the Euler relations formed by Auslander-Reiten triangles. We may regard $\mathscr{C}$ as the only possible $1$-cluster tilting subcategory of  $\mathscr{C}$. For a  2-Calabi-Yau triangulated category $\mathscr{C}$ with a $2$-cluster tilting subcategory $\mathcal{M}$,  Palu   \cite{Pa2}  proved that $K_{0}(\mathcal{T})$ is  isomorphic to $K_{0}^{\rm sp}(\mathcal{M})$ modulo  the relations formed by exchange triangles. Later, Fedele \cite{FE} gives a higher cluster tilting version  for  triangulated categories with a Serre functor. Let $\mathscr{C}$ be a Krull-Schmidt extriangulated category with an $n$-cluster tilting subcategory $\mathcal{M}$. Wang, Wei and Zhang \cite{Wl} proved that  $K_{0}(\mathscr{C})$  is isomorphic to the index Grothendieck group of $\mathcal{M}$. Besides, we note that the calculation of the Grothendieck groups for  cluster categories has been extensively studied, e.g. \cite{Bar, FE, Da}.
	
	These results inspire us to develop a general framework for the Grothendieck groups of extriangulated categories via higher rigid subcategories. For this purpose, we investigate the split Grothendieck group $K^{\rm sp}_{0}(\mathcal{M})$ of a $d$-rigid subcategory $\mathcal{M}$ in an extriangulated category $(\mathscr{C},\mathbb{E},\mathfrak{s})$. 
	Then we obtain several useful isomorphisms concerning $K^{\rm sp}_{0}(\mathcal{M})$ (see Theorem \ref{main}). In particular, our results provide an extriangulated and two-sided extension of the corresponding results in \cite{Ch}. As applications, we consider the case where $\mathcal{M}$ is a silting subcategory. Then we have the first main theorem of this paper.
	\begin{theoremi}\label{B}   \text{\rm (See Theorem \ref{main2} for  details)} 	
		Let $\mathcal{M}$ be a silting subcategory of  $\mathscr{C}$. Then 
		$$ K_{0}(\mathscr{C}) \cong K_{0}^{\rm sp}(\mathcal{M}).$$
	\end{theoremi}
	This result  extends \cite[Theorem 3.3]{Ch} from triangulated to extriangulated categories. We note  that Theorem \ref{B} does not require that $\mathscr{C}$ be Krull-Schmidt. It turns out that \cite[Theorem 4.1]{Ad2} holds for any extriangulated categories. Besides, we also have the  following relationship between $K_{0}(\mathscr{C})$ and bounded hereditary cotorsion pairs.
	\begin{corollaryi}\text{\rm (See Corollary \ref{C-5-7})} 	 Let $(\mathcal{T},\mathcal{F})$ be a bounded hereditary  cotorsion pair in $\mathscr{C}$. Then
		$$K_{0}(\mathcal{T})\cong K_{0}(\mathcal{F})\cong K_{0}(\mathscr{C}) \cong K_{0}^{\rm sp}(\mathcal{T}\cap\mathcal{F}).$$
	\end{corollaryi}
	
	We also consider the case where  $\mathcal{M}$ is a $d$-cluster tilting subcategory. By using various indices, we establish a homomorphism $\Phi:\mathcal{M}_{d}\cap\mathcal{M}^{1}\longrightarrow K_{0}^{\rm sp}(\mathcal{M})$. The index Grothendieck group $K_{0}^{\rm in}(\mathcal{M})$ is defined as the quotient group  $K_{0}^{\rm sp}(\mathcal{M})/\im\Phi$. The second main result of this paper is the following.
	\begin{theoremi}\label{C} \text{\rm (See Theorem \ref{main3} for details)} 	 Let $\mathcal{M}$ be a $d$-cluster tilting subcategory of $\mathscr{C}$. Then $$K_{0}(\mathscr{C})\cong K_{0}^{\rm in}(\mathcal{M}).$$
	\end{theoremi}
	Under  additional  assumptions, Theorem \ref{C} was proved in \cite[Theorem 3.13]{Wl}. It seems that $K_{0}(\mathscr{C})$ is isomorphic to a quotient group of $K_{0}^{\rm sp}(\mathcal{M})$ rather than $K_{0}^{\rm sp}(\mathcal{M})$. However, if we consider the relative extriangulated category $(\mathscr{C},\mathbb{E}_{\mathcal{M}},\mathfrak{s}|_{\mathbb{E}_{\mathcal{M}}})$ , then the Grothendieck group of $(\mathscr{C},\mathbb{E}_{\mathcal{M}},\mathfrak{s}|_{\mathbb{E}_{\mathcal{M}}})$ is isomorphic to   $K_{0}^{\rm sp}(\mathcal{M})$. We refer to Theorem \ref{main3} for more details.
	
	Finally, we compute the Grothendieck groups of $d$-cluster categories of type $A_n$ by applying Theorem \ref{C}. The following result can be regarded as a complete version of \cite[Proposition 7.21]{FE}.
	\begin{theoremi}\label{D} \text{\rm (See Theorem \ref{main5} for details)} Let $\mathcal{C}_{A_{n}}^{d}$ be the $d$-cluster category of type $A_{n}$. Then $$K_0(\mathcal{C}_{A_{n}}^{d})\cong\left\{
		\begin{aligned}
			\mathbb{Z}/(n+1)\mathbb{Z}, &~~\text{if $d$ is even,}\\
			\mathbb{Z}, &~~ \text{if $d$ and $n$ are odd,}\\
			0, & ~~\text{if $d$ is odd and $n$ is even}.
		\end{aligned}
		\right.
		$$
	\end{theoremi}
	\vspace{5mm}
	
	$\mathbf{Organization.}$ This paper is organized as follows. In Section 2, we summarize some fundamental definitions and properties of  extriangulated categories. In Section 3, we investigate the split  Grothendieck groups  of  $d$-rigid subcategories and obtain some useful  results. In Section 4, we apply our results on silting subcategories  and prove Theorem \ref{B}. Section 5 is concerned with the case of $d$-cluster tilting subcategories and then we prove Theorem \ref{C}. Finally, in Section 6,  we review the definition of $d$-cluster categories of type $A_{n}$ and their geometric model. Then we give a proof of Theorem \ref{D}.
	\vspace{5mm}
	
	$\mathbf{Conventions~and~Notation.}$ Throughout this paper, we assume that all considered categories are skeletally small  and that the subcategories are  closed under isomorphisms. For an additive category $\mathscr{C}$, we denote by $\ind(\mathscr{C})$ the set of isomorphism classes of indecomposable objects in $\mathscr{C}$. Given an object $X\in\mathscr{C}$, we denote by $\add(X)$ the full subcategory of $\mathscr{C}$ consisting of finite direct sums of direct summands of $X$.



	\section{Preliminaries}
	In this paper, we always assume that $\mathscr{C}:=(\mathscr{C},\mathbb{E},\mathfrak{s})$ is an extriangulated category. We omit the precise definition but refer to \cite[Section 2]{Na} for
	details.
	
	By Yoneda's lemma, any $\mathbb{E}$-extension $\delta\in \mathbb{E}(C,A)$ induces natural transformations
	$$\delta_{\sharp}: \mathscr{C}(-,C)\rightarrow\mathbb{E}(-,A)~~\text{and}~~\delta^{\sharp}: \mathscr{C}(A,-)\rightarrow\mathbb{E}(C,-).$$ For any $X\in\mathscr{C}$, these
	$(\delta_{\sharp})_X$ and $(\delta^{\sharp})_X$ are defined by
	$(\delta_{\sharp})_X:\mathscr{C}(X,C)\rightarrow\mathbb{E}(X,A), f\mapsto f^\ast\delta$
	and $(\delta^{\sharp})_X:\mathscr{C}(A,X)\rightarrow\mathbb{E}(C,X), g\mapsto g_\ast\delta.$ A sequence $$A\stackrel{x}{\longrightarrow}B\stackrel{y}{\longrightarrow}C\stackrel{\delta}\dashrightarrow$$ in $\mathscr{C}$ is called  an  {\em $\mathbb{E}$-triangle} if $\delta\in\mathbb{E}(C,A)$ and $\mathfrak{s}(\delta)=[A\stackrel{x}{\longrightarrow}B\stackrel{y}{\longrightarrow}C]$. Here $[A\stackrel{x}{\longrightarrow}B\stackrel{y}{\longrightarrow}C]$ is the equivalence class of $A\stackrel{x}{\longrightarrow}B\stackrel{y}{\longrightarrow}C$. We will write $\cocone(y):=A$ and $\cone(x):=C$ if necessary. 	We begin with recalling the following basic definitions.

	\begin{definition}  Let $\mathcal{X},\mathcal{Y}$ be two subcategories of $\mathscr{C}$.
		
		$(1)$ We denote by $\mathcal{X}*\mathcal{Y}$ the subcategory of $\mathscr{C}$ consisting of $M\in\mathscr{C}$ such that there is an $\mathbb{E}$-triangle
		$$X\stackrel{}{\longrightarrow}M\stackrel{}{\longrightarrow}Y\stackrel{}\dashrightarrow$$
		with $X\in \mathcal{X}$, $Y\in\mathcal{Y}$.
		We say $\mathcal{X}$ is {\em extension-closed} if $\mathcal{X}*\mathcal{X}\subseteq \mathcal{X}$.
		
		$(2)$ We denote by $\cone(\mathcal{X},\mathcal{Y})$ the subcategory of $\mathscr{C}$ consisting of $M\in\mathscr{C}$ such that there is an $\mathbb{E}$-triangle $$X\stackrel{}{\longrightarrow}Y\stackrel{}{\longrightarrow}M\stackrel{}\dashrightarrow$$
		with $X\in \mathcal{X}$, $Y\in\mathcal{Y}$. Dually, we define subcategory $\cocone(\mathcal{X},\mathcal{Y})$.
		
		$(3)$ We say $\mathcal{X}$ is {\em closed under cones} if $\cone(\mathcal{X},\mathcal{X})\subseteq \mathcal{X}.$ Dually, we say $\mathcal{X}$ is {\em closed under cocones} if $\cocone(\mathcal{X},\mathcal{X})\subseteq \mathcal{X}.$  
	\end{definition}

	An object $P\in\mathscr{C}$ is called {\em projective} if $\mathbb{E}(P,\mathscr{C})=0$. Dually, an object $I\in\mathscr{C}$ is called {\em injective} if $\mathbb{E}(\mathscr{C},I)=0$. We denote $\mathcal{P}(\mathscr{C})$ and $\mathcal{I}(\mathscr{C})$ the subcategories of projective objects and injective objects, respectively. We say $\mathscr{C}$ has enough projective objects if  $\mathscr{C}=\cone(\mathscr{C},\mathcal{P})$. Dually, we  say $\mathscr{C}$ has enough injective objects if  $\mathscr{C}=\cocone(\mathcal{I},\mathscr{C})$. 
	
	For $n>0$, the higher positive extensions $\mathbb{E}^n(-,-)$ in an extriangulated category have been defined in \cite{Go2}. If $\mathscr{C}$ has enough projective objects or enough injective objects, these $\mathbb{E}^n$ are isomorphic to those defined in \cite{Li} (cf. \cite[Remark 3.4]{Go2}). For a subcategory $\mathcal{X}$ of $\mathscr{C}$,  we define the full subcategory
	$$^{\perp_{[1,n]} }\mathcal{X}=\{M\in \mathscr{C}~|~\mathbb{E}^{i}(M,\mathcal{X})=0~\text{for}~i=1,2\cdots n \}.$$
	Dually, we define the full subcategory $\mathcal{X}^{\perp_{[1,n]} }$ in $\mathscr{C}$. The following proposition will be useful.
	\begin{proposition}$($\cite[Theorem 3.5]{Go}$)$ Let $A\stackrel{}{\longrightarrow}B\stackrel{}{\longrightarrow}C\stackrel{}\dashrightarrow$ be an $\mathbb{E}$-triangle in $\mathscr{C}$.
		
		$(1)$ We have a long exact sequence
		\begin{equation*}\begin{split}&\Hom(C,-)\rightarrow\Hom(B,-)\rightarrow\Hom(A,-)\rightarrow\mathbb{E}^{1}(C,-)\rightarrow\cdots\\
				&\cdots\rightarrow\mathbb{E}^{n-1}(A,-)\rightarrow\mathbb{E}^{n}(C,-)\rightarrow\mathbb{E}^{n}(B,-)\rightarrow\mathbb{E}^{n}(A,-)\rightarrow\cdots.\end{split}\end{equation*}
		
		$(2)$ We have a long exact sequence
		\begin{equation*}\begin{split}&\Hom(-,A)\rightarrow\Hom(-,B)\rightarrow\Hom(-,C)\rightarrow\mathbb{E}^{1}(-,A)\rightarrow\cdots\\
				&\cdots\rightarrow\mathbb{E}^{n-1}(-,C)\rightarrow\mathbb{E}^{n}(-,A)\rightarrow\mathbb{E}^{n}(-,B)\rightarrow\mathbb{E}^{n}(-,C)\rightarrow\cdots.\end{split}\end{equation*}
	\end{proposition}
	
	A finite sequence
	$$X_{n}\stackrel{d_{n}}{\longrightarrow}X_{n-1}\stackrel{d_{n-1}}{\longrightarrow}\cdots\stackrel{d_{2}}{\longrightarrow} X_{1}\stackrel{d_{1}}{\longrightarrow}X_{0}$$
	in $\mathscr{C}$ is said to be an
	{\em $\mathbb{E}$-triangle sequence}, if there exist $\mathbb{E}$-triangles $X_{n}\stackrel{d_{n}}{\longrightarrow}X_{n-1}\stackrel{f_{n-1}}{\longrightarrow}K_{n-1}\stackrel{}\dashrightarrow$,
	$K_{i+1}\stackrel{g_{i+1}}{\longrightarrow}X_{i}\stackrel{f_{i}}{\longrightarrow}K_{i}\stackrel{}\dashrightarrow,~~1<i<n-1$,
	and $K_{2}\stackrel{g_{2}}{\longrightarrow}X_{1}\stackrel{d_1}{\longrightarrow}X_0\stackrel{}\dashrightarrow$ such that $d_i=g_if_i$ for any $1<i<n$.  When we want to emphasize the objects $K_{2},\cdots,K_{n-1}$, we will write
	$$(X_{n}\stackrel{d_{n}}{\longrightarrow}X_{n-1}\stackrel{d_{n-1}}{\longrightarrow}\cdots \stackrel{d_{2}}{\longrightarrow} X_{1}\stackrel{d_{1}}{\longrightarrow}X_{0})\sim(K_{n-1},K_{n-2},~\cdots,K_{2}).$$
	
	\begin{remark}
		It should be noted that exact categories and triangulated categories are typical examples of extriangulated categories (cf. \cite[Example 2.13, Proposition 3.22]{Na}). Besides, there are many important extriangulated categories which are neither exact nor triangulated (cf. \cite{Go},\cite{Pa3}).
	\end{remark}

	\section{Indices and Grothendieck groups}
	In this paper, we fix an integer $d\geq1$. We start with introducing the following notations.
	
	\begin{definition}\label{D-3.1}
		Let $\mathcal{M}$ be an additive full subcategory of $\mathscr{C}$.
		\begin{enumerate}
			\item We say $\mathcal{M}$ is {\em $d$-rigid} if $\mathbb{E}^{i}(\mathcal{M},\mathcal{M})=0$ for any $1\leq i\leq d$.
			
			\item We say $\mathcal{M}$ is {\em $\infty$-rigid}  if $\mathbb{E}^{i}(\mathcal{M},\mathcal{M})=0$ for any $i\geq1$.
		\end{enumerate}
	\end{definition}
	
	It is obvious that a $\infty$-rigid subcategory is always $d$-rigid. There is a list of examples of $d$-rigid subcategories in various settings.
	
	\begin{example}\label{E-1} (1) It is clear that  $\mathcal{P}(\mathscr{C})$ and $\mathcal{I}(\mathscr{C})$ are $1$-rigid subcategories. Note that when $\mathscr{C}$ has enough projective objects (resp. injective objects), by the definition of higher extensions (cf. \cite[Section 5.1]{Li}), we can check that  $\mathcal{P}(\mathscr{C})$ (resp. $\mathcal{I}(\mathscr{C})$ ) is  $\infty$-rigid.
		
		$(2)$   Let $\mathcal{T}$ and  $\mathcal{F}$ be two subcategories of $\mathscr{C}$ which are closed under direct summands. We say $(\mathcal{T},\mathcal{F})$ is a {\em cotorsion pair} (cf. \cite[Definition 4.1]{Na}) if it satisfies the following conditions:
		\begin{enumerate}
			\item[(a)] $\mathbb{E}(\mathcal{T},\mathcal{F})=0.$
			\item[(b)] For any $X\in\mathscr{C}$, there exist two $\mathbb{E}$-triangles
			$$F_{1}\stackrel{}{\longrightarrow} T_{1}{\longrightarrow}X\stackrel{}\dashrightarrow~\text{and}~X\stackrel{}{\longrightarrow} F_{2}{\longrightarrow}T_{2}\stackrel{}\dashrightarrow$$
			such that $T_i\in \mathcal{T}$ and $F_i\in \mathcal{F}$.
		\end{enumerate}
		We say $(\mathcal{T},\mathcal{F})$ is  {\em hereditary} if $\mathbb{E}^{i}(\mathcal{T},\mathcal{F})=0$ for any $i\geq2$. It is obvious that $\mathcal{T}\cap\mathcal{F}$ is $1$-rigid and it is  $\infty$-rigid when $(\mathcal{T},\mathcal{F})$ is hereditary.

		$(3)$ Let $\mathcal{C}_{A_n}^{d}$ be the $d$-cluster tilting category of type $A_n$.
		In \cite{Ba}, the authors gave a geometric description of $\mathcal{C}_{A_n}^{d}$. By this, we can identify each indecomposable object in  $\mathcal{C}_{A_n}^{d}$ with a $d$-diagonal in a regular $(d(n+1)+2)$-gon. As shown in \cite[Proposition 2]{Th},  a subcategory $\mathcal{X}$ of $\mathcal{C}_{A_n}^{d}$ is $d$-rigid if and only if  $u$ does not cross $v$ for any $u,v\in\ind(\mathcal{X}).$ For instance, the Auslander-Reiten quiver of $\mathcal{C}_{A_3}^{2}$ is given by
		$$
		{\small
			
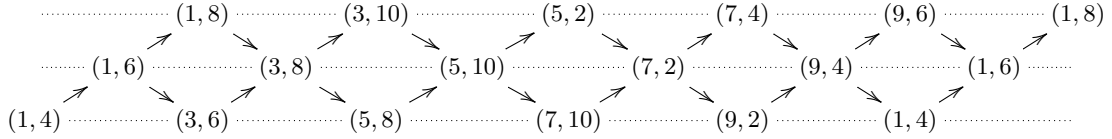
\xymatrix@-6.5mm{
					& &\ar@{.}[ll](1,8)\ar[rd]\ar@{.}[rr] & 
					& (3,10)\ar[rd]\ar@{.}[rr] & & (5,2)\ar[rd]\ar@{.}[rr] & & (7,4)\ar@{.}[rr] \ar[rd]  &&(9,6)\ar@{.}[rr]   \ar[rd]&&(1,8)\\
					\ar@{.}[r] 
					& (1,6)\ar[rd]\ar[ru]\ar@{.}[rr] &  
					& (3,8)\ar[rd]\ar[ru]\ar@{.}[rr] &  
					& (5,10)\ar[rd]\ar[ru]\ar@{.}[rr] & 
					& (7,2)\ar[ru]\ar[rd]\ar@{.}[rr] & & (9,4)\ar[ru]\ar[rd]\ar@{.}[rr]  &  & (1,6) \ar@{.}[r] \ar[ru] &\\
					(1,4)\ar[ru]\ar@{.}[rr] & 
					& (3,6)\ar[ru]\ar@{.}[rr] &  
					& (5,8)\ar[ru]\ar@{.}[rr] & & (7,10)\ar[ru] \ar@{.}[rr]& &(9,2)\ar@{.}[rr]\ar[ru]&&(1,4)\ar[ru]\ar@{.}[rr] &&
			}}
			$$	
			Then $\mathcal{X}=\add(\{(1,8),(7,2),(3,6)\})$ is 2-rigid.
		\end{example}

		\begin{definition} Let $\mathcal{M}$ be a subcategory of $\mathscr{C}$. For any $X\in\mathscr{C}$,  a {\em $\mathcal{M}^{L}$-filtration of length} $t\geq0$ for $X$ means an $\mathbb{E}$-triangle sequence
			\begin{equation*}
				X\stackrel{}{\longrightarrow}X_{0}\stackrel{}{\longrightarrow}\cdots \stackrel{}{\longrightarrow}X_{t-1}\stackrel{}{\longrightarrow}X_{t}{\longrightarrow}0
			\end{equation*}
			with each $X_{i}\in \mathcal{M}$. Dually, a {\em $\mathcal{M}^{R}$-filtration of length} $t\geq0$ for $X$ means an $\mathbb{E}$-triangle sequence
			\begin{equation*}
				0{\longrightarrow}Y_{t}\stackrel{}{\longrightarrow}Y_{t-1}\stackrel{}{\longrightarrow}\cdots \stackrel{}{\longrightarrow}Y_{0}\stackrel{}{\longrightarrow} X
			\end{equation*}
			such that each $Y_{i}\in\mathcal{M}$.
		\end{definition}
		Let $\mathcal{M}^{t}$ (resp. $\mathcal{M}_{t}$) denote the subcategory of $\mathscr{C}$ whose objects have $\mathcal{M}^{L}$-filtration (resp. $\mathcal{M}^{R}$-filtration) of length $t$. We define two subcategories $\mathcal{M}^{L}$ and $\mathcal{M}^{R}$ as
		$$\mathcal{M}^{L}=\bigcup\limits_{t\geq1}\mathcal{M}^{t}~\text{and}~\mathcal{M}^{R}=\bigcup\limits_{t\geq1}\mathcal{M}_{t}.$$
		By extending the $\mathbb{E}$-triangle sequences by zeros, we have $$\mathcal{M}=\mathcal{M}^{0}\subseteq\mathcal{M}^{1}\subseteq\cdots\subseteq\mathcal{M}^{L}~\text{and}~\mathcal{M}=\mathcal{M}_{0}\subseteq\mathcal{M}_{1}\subseteq\cdots \subseteq \mathcal{M}_{R}.$$
		
		\begin{remark}\label{R-3.4} Assume that $\mathscr{C}$ is a triangulated category with the suspension functor $\Sigma$. It is straightforward to check that
			$$ \mathcal{M}^{t}=\Sigma^{-t}\mathcal{M}\ast\cdots\ast\Sigma^{-1}\mathcal{M}\ast\mathcal{M}~\text{and}~\mathcal{M}_{t}=\mathcal{M}\ast\Sigma^{1}\mathcal{M}\ast\cdots\ast\Sigma^{t}\mathcal{M}.$$
			for any $t\geq0$. In particular, we have $\Sigma^{t}(\mathcal{M}^{t})=\mathcal{M}_{t}$.
		\end{remark}
		
		In what follows, we always assume that $\mathcal{M} $ is a $d$-rigid subcategory of $\mathscr{C}$.  The following observations  are useful.

		\begin{lemma}\label{L-A} Take $M\in\mathcal{M}$, $X\in\mathcal{M}^{t}$ and $Y\in\mathcal{M}_{t}$.
			
			$(1)$  If $d-t\geq1$, then $\mathbb{E}^{i}(X,M)=0$ for any $1\leq i\leq d-t$.
			
			$(2)$  If $d-t\geq1$, then $\mathbb{E}^{i}(M,Y)=0$ for any $1\leq i\leq d-t$.
			
			$(3)$ If $\mathcal{M}$ is $\infty$-rigid, then $\mathbb{E}^{i}(X,M)=0$ for any $i\geq1$.
			
			$(4)$ If $\mathcal{M}$ is $\infty$-rigid, then $\mathbb{E}^{i}(M,Y)=0$ for any $i\geq1$.
		\end{lemma}
		\begin{proof} (1) The case of $t=0$ is trivial. For the general case, we take a $\mathcal{M}^{L}$-filtration
			\begin{equation*}
				X\stackrel{f_{0}}{\longrightarrow}X_{0}\stackrel{}{\longrightarrow}\cdots \stackrel{}{\longrightarrow}X_{t-1}\stackrel{}{\longrightarrow}X_{t}{\longrightarrow} 0.
			\end{equation*}
			Applying the functor $\Hom(-,M)$ to the $\mathbb{E}$-triangle  $X\stackrel{}{\longrightarrow}X_{0}\stackrel{}{\longrightarrow}\cone(f_{0})\stackrel{}\dashrightarrow$. Then we obtain an exact sequence
			\begin{equation}\label{Eq-1}
				\mathbb{E}^{i}(X_{0},M)\longrightarrow\mathbb{E}^{i}(X,M)\longrightarrow\mathbb{E}^{i+1}(\cone(f_{0}),M)
			\end{equation}
			for any $i\geq1$. Note that $\mathbb{E}^{i}(X_{0},M)=0$ for any $1\leq i\leq d$. Since $\cone(f_{0})\in\mathcal{M}^{t-1}$, by induction $\mathbb{E}^{i}(\cone(f_{0}),M)=0$ for any $1\leq i\leq d-t+1$. By the exactness of (\ref{Eq-1}), we conclude that  $\mathbb{E}^{i}(X,M)=0$ for any $1\leq i\leq d-t$.  The proofs of $(2)$--$(4)$ are similar.
		\end{proof}

		Let $\mathcal{T}$ be a subcategory of $\mathscr{C}$.  The {\em split Grothendieck group $K^{\rm sp}_{0}(\mathcal{T})$} is defined by the quotient group $G(\mathcal{T})/\lr{[A\oplus B]-[A]-[B]~|~A,B\in\mathcal{T}}$, where $G(\mathcal{T})$ is the free abelian group on isomorphism classes $[T]$ of objects $T\in\mathcal{T}$. The {\em  Grothendieck group $K_{0}(\mathcal{T})$} is defined as
		$$K_{0}(\mathcal{T}):=K_{0}^{\rm sp}(\mathcal{T})/\lr{[A]-[B]+[C]~|~A\rightarrow B\rightarrow C\dashrightarrow$$
			$$\text{is an $\mathbb{E}$-triangle with terms in $\mathcal{T}$}}.$$
		
		\begin{lemma}\label{L-B}
			We take two $\mathcal{M}^{L}$-filtrations
			\begin{equation*}	(X\stackrel{}{\longrightarrow}X_{0}\stackrel{}{\longrightarrow}\cdots \stackrel{}{\longrightarrow}X_{t-1}\stackrel{}{\longrightarrow}X_{t}{\longrightarrow}0)\sim(K_{0},K_{1},\cdots,K_{t-1}\cong X_t)
			\end{equation*}
			\begin{equation*}	(X\stackrel{}{\longrightarrow}Y_{0}\stackrel{}{\longrightarrow}\cdots \stackrel{}{\longrightarrow}Y_{t-1}\stackrel{}{\longrightarrow}Y_{t}{\longrightarrow}0)\sim(P_{0},P_{1},\cdots,P_{t-1}\cong Y_t)
			\end{equation*}
			with $\mathbb{E}(K_{i},\mathcal{M})=\mathbb{E}(P_{i},\mathcal{M})=0$ for any $0\leq i\leq t-2$. Then	$$\sum\limits_{i=0}^{t}(-1)^{i}[X_{i}]=\sum\limits_{i=0}^{t}(-1)^{i}[Y_{i}]~\text{in}~K_{0}^{\rm sp}(\mathcal{M}).$$
		\end{lemma}
		\begin{proof} The case of $t=0$ is trivial. For the general case, we consider the following two $\mathbb{E}$-triangles $$X\stackrel{}{\longrightarrow}X_{0}\stackrel{}{\longrightarrow}K_{0}\stackrel{}\dashrightarrow~
			\text{and}~X\stackrel{}{\longrightarrow}Y_{0}\stackrel{}{\longrightarrow}P_{0}\stackrel{}\dashrightarrow.$$
			By \cite[Proposition 3.15]{Na}, there exists a commutative diagram of $\mathbb{E}$-triangles
			\begin{equation*}
				\xymatrix{
					X\ar[d]_{} \ar[r]^{} &X_{0} \ar[d]_{} \ar[r]^{} &K_{0} \ar@{=}[d]_{} \ar@{-->}[r]^{} &  \\
					Y_{0}\ar[d]_{} \ar[r]^{} &  Q\ar[d]_{} \ar[r]^{} &  K_{0}\ar@{-->}[r]^{0}&\\
					P_{0}\ar@{-->}[d]_{} \ar@{=}[r]^{} & P_{0} \ar@{-->}[d]^{0} &   &  \\
					&   &  &    }
			\end{equation*}
			By hypothesis, we have $Q\cong P_{0}\oplus X_{0}\cong K_{0}\oplus Y_{0}$. Then the $\mathbb{E}$-triangle $K_{0}\stackrel{}{\longrightarrow}X_{1}\stackrel{}{\longrightarrow}K_{1}\stackrel{}\dashrightarrow$ give rise to the following one
			$$
			K_{0}\oplus Y_{0}\stackrel{}{\longrightarrow}X_{1}\oplus Y_{0}\stackrel{}{\longrightarrow}K_{1}\stackrel{}\dashrightarrow.$$
			Therefore, we have a $\mathcal{M}^{L}$-filtration of length $t-1$ for $K_{0}\oplus Y_{0}$ as 
			\begin{equation*}
				K_{0}\oplus Y_{0}\stackrel{}{\longrightarrow}X_{1}\oplus Y_{0}\stackrel{}{\longrightarrow} X_{2}\stackrel{}{\longrightarrow} \cdots \stackrel{}{\longrightarrow}X_{t-1}\stackrel{}{\longrightarrow}X_{t}{\longrightarrow}0.
			\end{equation*}
			Similarly, we have  a $\mathcal{M}^{L}$-filtration of length $t-1$ for $P_{0}\oplus X_{0}$ as follows
			\begin{equation*}
				P_{0}\oplus X_{0}\stackrel{}{\longrightarrow}Y_{1}\oplus X_{0}\stackrel{}{\longrightarrow} Y_{2}\stackrel{}{\longrightarrow} \cdots \stackrel{}{\longrightarrow}Y_{t-1}\stackrel{}{\longrightarrow}Y_{t}{\longrightarrow}0.
			\end{equation*}
			By induction hypothesis, we have
			$$[X_{1}\oplus Y_{0}]+\Sigma_{i=2}^{t} (-1)^{i+1}[X_{i}]=[Y_{1}\oplus X_{0}]+\Sigma_{i=2}^{t} (-1)^{i+1}[Y_{i}]$$
			and hence $\sum\limits_{i=0}^{t}(-1)^{i}[X_{i}]=\sum\limits_{i=0}^{t}(-1)^{i}[Y_{i}]~\text{in}~K_{0}^{\rm sp}(\mathcal{M}).$
		\end{proof}
		
		Then we introduce the following definition.

		\begin{definition}\label{D-1}  For $X\in\mathcal{M}^{t}$, we take a $\mathcal{M}^{L}$-filtration
			\begin{equation*}
				X\stackrel{}{\longrightarrow}X_{0}\stackrel{}{\longrightarrow}\cdots \stackrel{}{\longrightarrow}X_{t-1}\stackrel{}{\longrightarrow}X_{t}{\longrightarrow}0
			\end{equation*}
			The {\em left index} of $X$ with respect to $\mathcal{M}$ is the element
			$$\ind^{ L}_{\mathcal{M}}(X)=\sum\limits_{i=0}^{t}(-1)^{i}[X_{i}]\in K_{0}^{\rm sp}(\mathcal{M}).$$
			Dually, if $X\in\mathcal{M}_{t}$, there is a $\mathcal{M}^{R}$-filtration
			\begin{equation*}
				0{\longrightarrow}Y_{t}\stackrel{}{\longrightarrow}Y_{t-1}\stackrel{}{\longrightarrow}\cdots \stackrel{}{\longrightarrow}Y_{0}\stackrel{}{\longrightarrow} X
			\end{equation*}
			The {\em right index} of $X$ with respect to $\mathcal{M}$ is the element
			$$\ind^{R}_{\mathcal{M}}(X)=\sum\limits_{i=0}^{t}(-1)^{i}[Y_{i}]\in K_{0}^{\rm sp}(\mathcal{M}).$$
		\end{definition}

		\begin{remark}\label{R-2-6} Suppose that $t\leq d$. Consider the following two $\mathbb{E}$-triangle sequences
			$$(X\stackrel{}{\longrightarrow}X_{0}\stackrel{}{\longrightarrow}\cdots \stackrel{}{\longrightarrow}X_{t-1}\stackrel{}{\longrightarrow}X_{t}{\longrightarrow}0)\sim(K_{0},K_{1},\cdots,K_{t-1}),$$
			$$(0{\longrightarrow}Y_{t}\stackrel{}{\longrightarrow}Y_{t-1}\stackrel{}{\longrightarrow}\cdots \stackrel{}{\longrightarrow}Y_{0}\stackrel{}{\longrightarrow} X)\sim(P_{t-1},\cdots,P_{1},P_{0}).$$
			Then Lemma \ref{L-A} implies that $\mathbb{E}(K_{i},\mathcal{M})=\mathbb{E}(\mathcal{M},P_{i})=0$ for any $1\leq i\leq t-1$. By using Lemma \ref{L-B} and its dual version, we infer that the elements $\ind^{ L}_{\mathcal{M}}(X)=\Sigma_{i=0}^{t}(-1)^{i}[X_{i}]$  and $\ind^{R}_{\mathcal{M}}(X)=\Sigma_{i=0}^{t}(-1)^{i}[Y_{i}]$ does not depend on the choice of $\mathbb{E}$-triangle sequence.  Hence, we can define two homomorphisms as follows:
			\begin{equation*}
				{\rm ind}^{\rm L}_{\mathcal{M}}:K_{0}^{\rm sp}(\mathcal{M}^{t})\longrightarrow K_{0}^{\rm sp}(\mathcal{M})~\text{and}~{\rm ind}^{\rm R}_{\mathcal{M}}:K_{0}^{\rm sp}(\mathcal{M}_{t})\longrightarrow K_{0}^{\rm sp}(\mathcal{M}).
			\end{equation*}
			If moreover $\mathcal{M}$ is $\infty$-rigid, then we have homomorphisms
			\begin{equation*}
				{\rm ind}^{\rm L}_{\mathcal{M}}:K_{0}^{\rm sp}(\mathcal{M}^{L})\longrightarrow K_{0}^{\rm sp}(\mathcal{M})~\text{and}~{\rm ind}^{\rm R}_{\mathcal{M}}:K_{0}^{\rm sp}(\mathcal{M}^{R})\longrightarrow K_{0}^{\rm sp}(\mathcal{M}).
			\end{equation*}
		\end{remark}
		
		We have the following useful characterization of indices.

		\begin{proposition}\label{L-7} Let $A\stackrel{}{\longrightarrow}B\stackrel{}{\longrightarrow}C\stackrel{}\dashrightarrow$ be an $\mathbb{E}$-triangle in $\mathscr{C}$. 
			
			$(1)$ Suppose that $A,C\in \mathcal{M}^{t}$ with $t\leq d$. If $t< d$ or  $\mathbb{E}(C,\mathcal{M})=0$, then $B\in\mathcal{M}^{t}$ and 
			\begin{equation*}
				\ind^{ L}_{\mathcal{M}}(B)=\ind^{ L}_{\mathcal{M}}(A)+\ind^{ L}_{\mathcal{M}}(C).
			\end{equation*}
			
			$(2)$ Suppose that  $A,C\in \mathcal{M}_{t}$ with $t\leq d$. If $t<d$ or  $\mathbb{E}(\mathcal{M},A)=0$, then $B\in\mathcal{M}_{t}$ and 
			\begin{equation*}
				\ind^{R}_{\mathcal{M}}(B)=\ind^{ R}_{\mathcal{M}}(A)+\ind^{ R}_{\mathcal{M}}(C).
			\end{equation*}
			
		\end{proposition}
		\begin{proof} 	Suppose that		\begin{equation*}
				(A\stackrel{}{\longrightarrow}X_{0}\stackrel{}{\longrightarrow}\cdots \stackrel{}{\longrightarrow}X_{t-1}\stackrel{}{\longrightarrow}X_{t}{\longrightarrow}0)\sim(K_{0},K_{1},\cdots,K_{t-1}\cong X_t),
			\end{equation*}
			\begin{equation*}
				(C\stackrel{}{\longrightarrow}Y_{0}\stackrel{}{\longrightarrow}\cdots \stackrel{}{\longrightarrow}Y_{t-1}\stackrel{}{\longrightarrow}Y_{t}{\longrightarrow}0)\sim (P_{0},P_{1},\cdots,P_{t-1}\cong Y_t).
			\end{equation*}
			We consider the following two cases:
			
			$(\mathbf{Case~1})$ $t<d.$ We claim that  $B$ has a $\mathcal{M}^{L}$-filtration of length $t$ as follows:
			\begin{equation}\label{E-3}
				B\stackrel{}{\longrightarrow}X_{0}\oplus Y_{0}\stackrel{}{\longrightarrow}\cdots \stackrel{}{\longrightarrow}X_{t-1}\oplus Y_{t-1}\stackrel{}{\longrightarrow}X_{t}\oplus Y_{t}{\longrightarrow}0.
			\end{equation}
			Using Lemma \ref{L-A}(1) and the dual  of \cite[Lemma 4.14]{Hu}, there exists a commutative diagram of $\mathbb{E}$-triangles
			\begin{equation}\label{G-1}
				\xymatrix{
					A\ar[r]^{} \ar[d]_{}& B \ar[r]^{}\ar[d] &  C\ar@{-->}[r]^{} \ar[d]_{} &  \\
					X_{0} \ar[r]^{} \ar[d]&  X_{0}\oplus Y_{0} \ar[r]^{}\ar[d] &  Y_{0}\ar@{-->}[r]^{0}\ar[d] &  \\
					K_{0}\ar[r]^{}\ar@{-->}[d]^{} &  Q_{0} \ar[r]^{} \ar@{-->}[d]^{}& P_{0}\ar@{-->}[r]^{}\ar@{-->}[d]^{}&\\
					&  &  & . }
			\end{equation}
			Now we repeat the argument to the $\mathbb{E}$-triangle  $K_{0}\stackrel{}{\longrightarrow}Q_{0}\stackrel{}{\longrightarrow}P_{0}\stackrel{}\dashrightarrow$.  Then we obtain an $\mathbb{E}$-triangle sequence
			\begin{equation*}
				B\stackrel{}{\longrightarrow}X_{0}\oplus Y_{0}\stackrel{}{\longrightarrow}X_{1}\oplus Y_{1}\rightarrow Q_{1}
			\end{equation*}
			such that $Q_{0}\stackrel{}{\longrightarrow}X_{1}\oplus Y_{1}\stackrel{}{\longrightarrow}Q_{1}\stackrel{}\dashrightarrow$ and $K_{1}\stackrel{}{\longrightarrow} Q_{1}{\longrightarrow}P_{1}\stackrel{}\dashrightarrow$ are $\mathbb{E}$-triangles. By repeating this process, we have an $\mathbb{E}$-triangle sequence
			\begin{equation*}
				B\stackrel{}{\longrightarrow}X_{0}\oplus Y_{0}\stackrel{}{\longrightarrow}\cdots \stackrel{}{\longrightarrow}X_{t-1}\oplus Y_{t-1}\stackrel{}{\longrightarrow}Q_{t-1}
			\end{equation*}
			such that $Q_{t-2}\stackrel{}{\longrightarrow}X_{t-1}\oplus Y_{t-1}\stackrel{}{\longrightarrow}Q_{t-1}\stackrel{}\dashrightarrow$ and $K_{t-1}\stackrel{}{\longrightarrow} Q_{t-1}{\longrightarrow}P_{t-1}\stackrel{}\dashrightarrow$ are $\mathbb{E}$-triangles. Since $\mathbb{E}(P_{t-1},K_{t-1})=\mathbb{E}(Y_{t},X_{t})=0$, we infer that $Q_{t-1}\cong X_{t}\oplus Y_{t}\in \mathcal{M}$.
			
			$(\mathbf{Case~2})$  $\mathbb{E}(C,\mathcal{M})=0$. By hypothesis, there exists a  commutative diagram of $\mathbb{E}$-triangles as (\ref{G-1}). Note that
			$$K_{0}\stackrel{}{\longrightarrow}Q_{0}\stackrel{}{\longrightarrow}P_{0}\stackrel{}\dashrightarrow$$
			is an $\mathbb{E}$-triangle with $K_{0},P_{0}\in\mathcal{M}^{t-1}$. By (Case 1), there exists an $\mathbb{E}$-triangle sequence
			\begin{equation}\label{E-33}
				Q_{0}\stackrel{}{\longrightarrow}X_{1}\oplus Y_{1}\stackrel{}{\longrightarrow}\cdots \stackrel{}{\longrightarrow}X_{t-1}\oplus Y_{t-1}\stackrel{}{\longrightarrow}X_{t}\oplus Y_{t}{\longrightarrow}0.
			\end{equation}
			Combining (\ref{E-33}) with the $\mathbb{E}$-triangle $B\stackrel{}{\longrightarrow}X_{0}\oplus Y_{0}\stackrel{}{\longrightarrow}Q_{0}\stackrel{}\dashrightarrow$, we obtain an $\mathcal{M}^{L}$-filtration as (\ref{E-3}).
			
			This finishes the proof of (1). The proof of (2) is similar.
		\end{proof}
		
		\begin{remark}\label{R-3-9} (1) If $t<d$, then Proposition \ref{L-7}(1) implies that $\mathcal{M}^{t}$ is closed under extensions.  However, $\mathcal{M}^{d}$ is not  necessarily closed under extensions. For example, we take $\mathcal{M}=\add(\{(3,10),(9,4)\})$ in $\mathcal{C}_{A_3}^{2}$. In this case, $\mathcal{M}$ is closed under direct summands. As stated in Example \ref{E-1}(3), the subcategory $\mathcal{M}$ is $2$-rigid. 	Observe that  there is a triangle
			$$(9,2)\longrightarrow (9,6)\longrightarrow (1,6)\longrightarrow \Sigma (9,2).$$
			with 	 $(9,2)\in (5,2)*(9,4)\in\Sigma^{-2}\mathcal{M}\ast \mathcal{M}\subseteq\mathcal{M}^{2}$ and $(1,6)\in\Sigma^{-2}\mathcal{M}\subseteq\mathcal{M}^{2}$. 
			We claim that $(9,6)\notin \mathcal{M}^{2}$. Otherwise, there exist two triangles
			$$(9,6)\longrightarrow M_0\longrightarrow X\longrightarrow \Sigma (9,6)~\text{and}~X\longrightarrow M_1\longrightarrow M_2\longrightarrow \Sigma X$$
			such that each $M_i\in\mathcal{M}$. Since $\Hom((9,6),\mathcal{M})=0$, we get $X\cong M_0\oplus (5,8)$. We can check that $\Ext^{1}(\mathcal{M},X)=0$ and hence  $(5,8)\in\mathcal{M}$. This  is a contradiction.
			
			$(2)$ Suppose that $t<d$. By using Remark \ref{R-2-6} and  Lemma \ref{L-7}, we consider the following homomorphisms 	
			\begin{equation*}
				{\rm ind}^{\rm L}_{\mathcal{M}}:K_{0}(\mathcal{M}^{t})\longrightarrow K_{0}^{\rm sp}(\mathcal{M})~\text{and}~{\rm ind}^{\rm R}_{\mathcal{M}}:K_{0}^{\rm sp}(\mathcal{M}_{t})\longrightarrow K_{0}^{\rm sp}(\mathcal{M}).
			\end{equation*}
			If $\mathcal{M}$ is $\infty$-rigid, then we have
			\begin{equation*}
				{\rm ind}^{\rm L}_{\mathcal{M}}:K_{0}(\mathcal{M}^{L})\longrightarrow K_{0}^{\rm sp}(\mathcal{M})~\text{and}~{\rm ind}^{\rm R}_{\mathcal{M}}:K_{0}^{\rm sp}(\mathcal{M}_{t})\longrightarrow K_{0}^{\rm sp}(\mathcal{M}).
			\end{equation*}
		\end{remark}

		Let $\mathcal{T}$ be an additive category.  A {\em right $\mathcal{T}$-module} is a  contravariant linear functor of the form $F:\mathcal{T}\longrightarrow \Mod \mathbb{Z}$. The category $\Mod \mathcal{T}$ of right $\mathcal{T}$-modules is an abelian category. By Yoneda’s lemma, representable functors are projective objects in  $\Mod \mathcal{T}$. A right $\mathcal{T}$-module $F$ is called {\em coherent} if there exists an exact sequence
		$$(-,T_{1})\longrightarrow(-,T_{0})\longrightarrow F\longrightarrow0$$
		such that each $T_{i}\in\mathcal{T}.$ We denote by $\mod \mathcal{T}$ the subcategory of  $\Mod \mathcal{T}$ whose objects  are the coherent right $\mathcal{T}$-modules. Note that $\mod \mathcal{M}$ is closed under extensions and cokernels in  $\Mod \mathcal{T}$. Dually, one can define categories $\Mod \mathcal{T}^{\rm op}$ and  $\mod \mathcal{T}^{\rm op}$. 
		Consider the following two homological functors 
		$$F_{\mathcal{M}}:\mathscr{C}\longrightarrow \Mod\mathcal{M};~X\mapsto\Hom(-,X)|_{\mathcal{M}},$$
		$$G_{\mathcal{M}}:\mathscr{C}\longrightarrow \Mod \mathcal{M}^{\rm op};~X\mapsto\Hom(X,-)|_{\mathcal{M}}.$$

		\begin{lemma}\label{L-3.10}Take $X\in\mathscr{C}$.
			
			$(1)$ If  $X\in \mathcal{M}_{d}$, then $F_{\mathcal{M}}(X)\in \mod\mathcal{M}$.		
			
			$(2)$ If  $X\in \mathcal{M}^{d}$, then $G_{\mathcal{M}}(X)\in \mod\mathcal{M}^{\rm op} $.
			
			$(3)$ If $X\in\mathcal{M}^{1}$, then $\mathbb{E}(-,X)|_{\mathcal{M}}\in\mod\mathcal{M}$.
			
			$(4)$ If $X\in\mathcal{M}_{1}$, then $\mathbb{E}(X,-)|_{\mathcal{M}}\in\mod\mathcal{M}^{\rm op}$.

		\end{lemma}

		\begin{proof}
			(1) Take a $\mathcal{M}^{R}$-filtration
			$$(0{\longrightarrow}M_{d}\stackrel{}{\longrightarrow}M_{d-1}\stackrel{}{\longrightarrow}\cdots \stackrel{}{\longrightarrow}M_{0}\stackrel{}{\longrightarrow} X)\sim(P_{d-1},\cdots,P_{1},P_{0}).$$
			The case of $d=0$ is trivial.  Applying the functor $F_{\mathcal{M}}$ to the $\mathbb{E}$-triangle $P_{0}\stackrel{}{\longrightarrow} M_0{\longrightarrow}X\stackrel{}\dashrightarrow$, we obtain an exact sequence
			$$F_{\mathcal{M}}(P_0)\longrightarrow F_{\mathcal{M}}(M_0)\longrightarrow F_{\mathcal{M}}(X)\longrightarrow \mathbb{E}(-,P_0)|_{\mathcal{M}}=0.$$
			Since $P_0\in\mathcal{M}_{d-1}$, by induction $F_{\mathcal{M}}(P_0)\in \mod\mathcal{M}$. Recall that $\mod\mathcal{M}$ is closed under cokernels. It follows that $F_{\mathcal{M}}(X)\in \mod\mathcal{M}$. The proof of (2) is similar.
			
			$(3)$ By hypothesis, there exists an $\mathbb{E}$-triangle
			$X{\longrightarrow}M_0{\longrightarrow}M_1\stackrel{}\dashrightarrow$ such that each $M_i\in\mathcal{M}$. Applying the functor $F_{\mathcal{M}}$ to these  $\mathbb{E}$-triangle, we have an exact sequence
			$$F_{\mathcal{M}}(M_0)\longrightarrow F_{\mathcal{M}}(M_1)\longrightarrow \mathbb{E}(-,X)|_{\mathcal{M}}\longrightarrow \mathbb{E}(-,M_0)|_{\mathcal{M}}=0.$$  
			It means that   $\mathbb{E}(-,X)|_{\mathcal{M}}\in\mod\mathcal{M}$. The proof of (4) is similar.
		\end{proof}

		\begin{definition}  By Lemma \ref{L-3.10},  we have two functors:
			$$\mathbb{E}(-,?)|_{\mathcal{M}}:\mathcal{M}_{d}\cap\mathcal{M}^{1}\longrightarrow\mod \mathcal{M};~X\mapsto \mathbb{E}(-,X)|_{\mathcal{M}},$$
			$$\mathbb{E}(?,-)|_{\mathcal{M}}:\mathcal{M}^{d}\cap\mathcal{M}_{1}\longrightarrow\mod \mathcal{M}^{\rm op};~X\mapsto \mathbb{E}(X,-)|_{\mathcal{M}}.$$
			We define
			$$\Phi:\mathcal{M}_{d}\cap\mathcal{M}^{1}\longrightarrow K_{0}^{\rm sp}(\mathcal{M});X\mapsto \ind^{R}_{\mathcal{M}}(X)-\ind^{L}_{\mathcal{M}}(X),$$
			$$\Psi:\mathcal{M}^{d}\cap\mathcal{M}_{1}\longrightarrow K_{0}^{\rm sp}(\mathcal{M});X\mapsto \ind^{L}_{\mathcal{M}}(X)-\ind^{R}_{\mathcal{M}}(X).$$
		\end{definition}
		
		We have the following characterizations of indices on $\mathbb{E}$-triangles.
		\begin{proposition}\label{L-3-11} Let $A\stackrel{f}{\longrightarrow}B\stackrel{g}{\longrightarrow}C\stackrel{}\dashrightarrow$ be an $\mathbb{E}$-triangle in $\mathscr{C}$.
			
			$(1)$ If $A,B,C\in \mathcal{M}^{d}$, then $$\coker(G_{\mathcal{M}}(f))=\mathbb{E}(S,-)|_{\mathcal{M}}\in\im(\mathbb{E}(?,-)|_{\mathcal{M}})$$ for some $S\in\mathcal{M}^{d}\cap\mathcal{M}_{1}$. Moreover, we have
			$$\ind^{L}_{\mathcal{M}}(A)-\ind^{ L}_{\mathcal{M}}(B)+\ind^{ L}_{\mathcal{M}}(C)=\Psi(S).
			$$
			
			$(2)$ If $A,B,C\in \mathcal{M}_{d}$, then $$\coker(F_{\mathcal{M}}(g))=\mathbb{E}(-,S)|_{\mathcal{M}}\in\im(\mathbb{E}(-,?)|_{\mathcal{M}})$$
			for some $S\in\mathcal{M}_{d}\cap\mathcal{M}^{1}$. Moreover, we have
			$$\ind^{R}_{\mathcal{M}}(A)-\ind^{ R}_{\mathcal{M}}(B)+\ind^{ R}_{\mathcal{M}}(C)=\Phi(S).$$
		\end{proposition}

		\begin{proof}
			(1)  Suppose that		\begin{equation*}
				(A\stackrel{}{\longrightarrow}X_{0}\stackrel{}{\longrightarrow}\cdots \stackrel{}{\longrightarrow}X_{d-1}\stackrel{}{\longrightarrow}X_{d}{\longrightarrow}0)\sim(K_{0},K_{1},\cdots,K_{d-1}\cong X_d),
			\end{equation*}
			\begin{equation*}
				(C\stackrel{}{\longrightarrow}Y_{0}\stackrel{}{\longrightarrow}\cdots \stackrel{}{\longrightarrow}Y_{d-1}\stackrel{}{\longrightarrow}Y_{d}{\longrightarrow}0)\sim (P_{0},P_{1},\cdots,P_{d-1}\cong Y_d).
			\end{equation*}
			Now, we consider the following commutative diagram of $\mathbb{E}$-triangles
			\begin{equation}\label{D-1}
				\xymatrix{
					A\ar[r]^{f} \ar[d]_{}& B \ar[r]^{g}\ar[d] &  C\ar@{-->}[r]^{} \ar@{=}[d]_{} &  \\
					X_{0} \ar[r]^{a} \ar[d]&  T_{1} \ar[r]^{b}\ar[d] &  C\ar@{-->}[r]^{}&  \\
					K_{0}\ar@{=}[r]^{}\ar@{-->}[d]^{} &  K_{0} \ar@{-->}[d]^{}& &\\
					&  &  &  }
			\end{equation}
			We divide
			the proof into the following steps:
			
			$\mathbf{(Step~1)}$ Since $K_{0}\in\mathcal{M}^{d-1}$ and $B\in\mathcal{M}^{d}$, it follows that $T_{1}\in \mathcal{M}^{d}$ and $\ind^{L}_{\mathcal{M}}(T_{1})=\ind^{ L}_{\mathcal{M}}(B)+\ind^{ L}_{\mathcal{M}}(K_{0})$ by Proposition \ref{L-7}(1). On the other hand, we have  $[X_{0}]=\ind^{L}_{\mathcal{M}}(A)+\ind^{L}_{\mathcal{M}}(K_{0})$. Hence,
			\begin{flalign*}
				\ind^{ L}_{\mathcal{M}}(A)-\ind^{ L}_{\mathcal{M}}(B)+\ind^{ L}_{\mathcal{M}}(C)&=[X_{0}]-\ind^{ L}_{\mathcal{M}}(T_{1})+\ind^{ L}_{\mathcal{M}}(C).
			\end{flalign*}
			
			$\mathbf{(Step~2)}$ Applying the functor $G_{\mathcal{M}}$ to  (\ref{D-1}),  we have a  commutative diagram of exact sequences as follows:
			\begin{equation*}
				\xymatrix{
					&	&	& \mathbb{E}(K_0,-)|_{\mathcal{M}}=0 \ar[d]      \\
					G_{\mathcal{M}}(T_1)\ar[r]^{G_{\mathcal{M}}(a)}\ar[d]	&	G_{\mathcal{M}}(X_0)\ar[r]^{}	\ar[d]&	\mathbb{E}(C,-)|_{\mathcal{M}} \ar[r]^{b^{*}} \ar@{=}[d]&  	\mathbb{E}(T_1,-)|_{\mathcal{M}} \ar[d]   	  \\
					G_{\mathcal{M}}(B)\ar[r]^{G_{\mathcal{M}}(f)}		&	G_{\mathcal{M}}(A)\ar[r]^{}		&	\mathbb{E}(C,-)|_{\mathcal{M}}\ar[r]^{g^{*}}	 & 	\mathbb{E}(B,-)|_{\mathcal{M}}  	}
			\end{equation*}
			Note that $\mathbb{E}(K_0,-)|_{\mathcal{M}}=0$. It follows that 
			$$\coker(G_{\mathcal{M}}(f))\cong\ker(g^{*})=\ker(b^{*})\cong\coker(G_{\mathcal{M}}(a)).$$
			
			$\mathbf{(Step~3)}$ Recall that $T_{1}\in \mathcal{M}^{d}$. Then there exists an $\mathbb{E}$-triangle
			$$T_{1}\stackrel{}{\longrightarrow} M{\longrightarrow}T_{2}\stackrel{}\dashrightarrow$$
			such that $M\in\mathcal{M}$ and $T_{2}\in \mathcal{M}^{d-1}$. In particular, we have $[M]=\ind^{ L}_{\mathcal{M}}(T_{1})+\ind^{ L}_{\mathcal{M}}(T_{2})$. Applying (ET4), we have the following commutative diagram
			\begin{equation}\label{D-2}
				\xymatrix{
					X_{0}\ar[r]^{a} \ar@{=}[d]_{}& T_{1} \ar[r]^{b}\ar[d] &  C\ar@{-->}[r]^{} \ar[d]_{} &  \\
					X_{0} \ar[r]^{} &  M \ar[r]^{c}\ar[d] &  S\ar@{-->}[r]^{}\ar[d] &  \\
					&  T_{2} \ar@{=}[r] \ar@{-->}[d]^{}& T_{2}\ar@{-->}[d]^{}&\\
					&  &  &  }
			\end{equation}
			
			Observe that $S\in\mathcal{M}_{1}$ and $\ind^{R}_{\mathcal{M}}(S)=[M]-[X_{0}]$. By Proposition \ref{L-7}(1), we also have $S\in\mathcal{M}^{d}$ and  $\ind^{L}_{\mathcal{M}}(S)=\ind^{ L}_{\mathcal{M}}(C)+\ind^{ L}_{\mathcal{M}}(T_{2})$. It follows that
			\begin{flalign*}
				\ind^{ L}_{\mathcal{M}}(A)-\ind^{ L}_{\mathcal{M}}(B)+\ind^{ L}_{\mathcal{M}}(C)&=[X_{0}]-\ind^{L}_{\mathcal{M}}(T_{1})+\ind^{ L}_{\mathcal{M}}(C)\\
				&=[X_{0}]-[M]+\ind^{ L}_{\mathcal{M}}(S)\\
				&=\ind^{L}_{\mathcal{M}}(S)-\ind^{R}_{\mathcal{M}}(S)\\
				&=\Psi(S).
			\end{flalign*}
			
			$\mathbf{(Step~4)}$ 	Applying the functor $G_{\mathcal{M}}$ to  (\ref{D-2}),  we have a  commutative diagram of exact sequences 
			\begin{equation*}
				\xymatrix{
					&	& \mathbb{E}(T_2,-)|_{\mathcal{M}}=0 \ar[d]  &    \\
					G_{\mathcal{M}}(M)\ar[r]^{G_{\mathcal{M}}(c)}\ar[d]	&	G_{\mathcal{M}}(X_0)\ar[r]^{}	\ar@{=}[d]&	\mathbb{E}(S,-)|_{\mathcal{M}}  \ar[d]\ar[r]^-{c^{*}}	  & \mathbb{E}(M,-)|_{\mathcal{M}}=0 \ar[d]\\
					G_{\mathcal{M}}(T_1)\ar[r]^{G_{\mathcal{M}}(a)}		&	G_{\mathcal{M}}(X_0)\ar[r]^{}		&	\mathbb{E}(C,-)|_{\mathcal{M}}\ar[r]^-{b^{*}}	   &\mathbb{E}(T_1,-)|_{\mathcal{M}}	}
			\end{equation*}
			By (Step 2), we infer that
			$$\coker(G_{\mathcal{M}}(f))\cong\coker(G_{\mathcal{M}}(a))\cong\ker(b^{*})=\ker(c^{*})=\mathbb{E}(S,-)|_{\mathcal{M}}.$$
			This finishes the proof. The proof of (2) is similar.
		\end{proof}

		\begin{remark}\label{R-3-13} For any $S\in\mathcal{M}^{d}\cap\mathcal{M}_{1}$, there exists an $\mathbb{E}$-triangle sequence
			\begin{equation*}
				X_{0}\stackrel{f}{\longrightarrow}X_{1}\stackrel{}{\longrightarrow}\cdots \stackrel{}{\longrightarrow}X_{d+2}\stackrel{}{\longrightarrow}X_{d+3}
			\end{equation*}
			such that each $X_{i}\in \mathcal{M}$ and $\cone(f)\cong S$. This implies that
			$$K_{0}^{\rm sp}(\mathcal{M})/\im\Psi=K_{0}^{\rm sp}(\mathcal{M})/\lr{\sum_{i=0}^{d+2}(-1)^{i}[X_{i}]~|~X_{0}{\longrightarrow}\cdots \stackrel{}{\longrightarrow} X_{d+1}\stackrel{}{\longrightarrow}X_{d+3}~$$
				$$\text{is an $\mathbb{E}$-triangle sequence with terms in  $\mathcal{M}$}}.$$
			We now repeat the argument to $\mathcal{M}_{d}\cap\mathcal{M}^{1}$. Then we obtain $K_{0}^{\rm sp}(\mathcal{M})/\im\Phi=K_{0}^{\rm sp}(\mathcal{M})/\im\Psi$.
		\end{remark}
		Following \cite{Wl}, the {\em index Grothendieck group} of $\mathcal{M}$ is defined as 
		$$ K_{0}^{\rm in}(\mathcal{M}):=K_{0}^{\rm sp}(\mathcal{M})/\im\Phi=K_{0}^{\rm sp}(\mathcal{M})/\im\Psi.$$
		Note that $K_{0}^{\rm in}(\mathcal{M})=K_{0}^{\rm sp}(\mathcal{M})$ if $\mathcal{M}$ is $\infty$-rigid. Let us state the main result in this section in the following theorem.

		\begin{theorem}\label{main} Let $\mathcal{M}$ be a $d$-rigid subcategory of $\mathscr{C}$.
			
			$(1)$  If $1\leq t<d$, then there exist  group isomorphisms
			\[
			\begin{tikzcd}[column sep =40.5, row sep =40.5]
				K_{0}(\mathcal{M}^{t})  \rar[rightarrow,"\ind^{L}_{\mathcal{M}}","\cong"']  &   K_{0}^{\rm sp}(\mathcal{M})&K_{0}(\mathcal{M}_{t}). \lar[rightarrow,"\ind^{R}_{\mathcal{M}}"',"\cong"]
			\end{tikzcd}
			\]
			
			$(2)$  We have group isomorphisms
			\[
			\begin{tikzcd}[column sep =40.5, row sep =40.5]
				K_{0}(\mathcal{M}^{d})  \rar[rightarrow,"\ind^{L}_{\mathcal{M}}","\cong"']  &   K_{0}^{\rm in}(\mathcal{M})
				&  	K_{0}(\mathcal{M}_{d})  \lar[rightarrow,"\ind^{R}_{\mathcal{M}}"',"\cong"].
			\end{tikzcd}
			\]		
			
			$(3)$  If $\mathcal{M}$ is $\infty$-rigid, then there exist group isomorphisms
			\[
			\begin{tikzcd}[column sep =40.5, row sep =40.5]
				K_{0}(\mathcal{M}^{L})  \rar[rightarrow,"\ind^{L}_{\mathcal{M}}","\cong"']  &   K_{0}^{\rm sp}(\mathcal{M})&K_{0}(\mathcal{M}^{R}). \lar[rightarrow,"\ind^{R}_{\mathcal{M}}"',"\cong"]
			\end{tikzcd}
			\]
		\end{theorem}
		\begin{proof} (1) By Remark \ref{R-3-9}(2),   the map $\ind^{L}_{\mathcal{M}}:K_{0}(\mathcal{M}^{t})\longrightarrow K_{0}^{\rm sp}(\mathcal{M})$ is well-defined. For any $[X]\in K_{0}(\mathcal{M}^{t})$, we have $\ind^{L}_{\mathcal{M}}(X)=[X]$ in $K_{0}(\mathcal{M})$. On the other hand, we have $\ind^{L}_{\mathcal{M}}(M)=[M]$ for any $M\in \mathcal{M}$. Hence,  $\ind^{L}_{\mathcal{M}}$ is inverse to the natural inclusion $ K_{0}^{\rm sp}(\mathcal{M})\hookrightarrow K_{0}(\mathcal{M}^{t})$. Similarly, one can check that $\ind^{R}_{\mathcal{M}}:K_{0}(\mathcal{M}_{t})\longrightarrow K_{0}^{\rm sp}(\mathcal{M})$ is  an isomorphism.
			
			$(2)$ Note that $\ind^{L}_{\mathcal{M}}:K_{0}(\mathcal{M}_{d}) \longrightarrow  K_{0}^{\rm sp}(\mathcal{M})/\im\Psi$ is well-defined follows from Proposition \ref{L-3-11}(1). We consider the following diagram
			\[
			\begin{tikzcd}[column sep =40.5, row sep =40.5]
				K^{\rm sp}_{0}(\mathcal{M})  \dar[hookrightarrow,"i"'] \rar[twoheadrightarrow,"\pi_{1}"]  &    K_{0}^{\rm sp}(\mathcal{M})/\im\Psi \dar["l", dashed] \\
				K^{\rm sp}_{0}(\mathcal{M}^{d}) \rar[twoheadrightarrow,"\pi_{2}"'] & K_{0}(\mathcal{M}^{d})  
			\end{tikzcd}
			\]
			where $i$ is an inclusion and $\pi_{1},\pi_{2}$ are natural projections. For any $S\in\mathcal{M}^{d}\cap\mathcal{M}_{1}$, we have an $\mathbb{E}$-triangle sequence
			\begin{equation*}
				X_{0}\stackrel{f}{\longrightarrow}X_{1}\stackrel{}{\longrightarrow}X_{2}{\longrightarrow} X_{3}{\longrightarrow}\cdots \stackrel{}{\longrightarrow}X_{t+1}\stackrel{}{\longrightarrow}X_{t+2}
			\end{equation*}
			such that $\cone(f)=S$ and each $X_{i}\in \mathcal{M}$.  Hence $\ind^{L}_{\mathcal{M}}(S)-\ind^{R}_{\mathcal{M}}(S)=\Sigma^{t+2}_{i=0}(-1)^{i}[X_{i}]=0$ in  $K_{0}(\mathcal{M})$. The observation above implies that the inclusion $i$ induces a group homomorphism $$l: K_{0}^{\rm sp}(\mathcal{M})/\im\Psi \rightarrow K_{0}(\mathcal{M}^{d})$$ such that $l\pi_{1}=\pi_{2}i$. It is routine to verify that $\ind^{L}_{\mathcal{M}}$ is inverse to the group homomorphism $l$. A similar argument shows that
			$\ind^{R}_{\mathcal{M}}:	K_{0}(\mathcal{M}_{d})\longrightarrow K_{0}^{\rm sp}(\mathcal{M})/\im\Phi$ is an isomorphism.
			
			$(3)$ Using the analogous arguments as those proving (1),  we can prove  $\ind^{L}_{\mathcal{M}}:	K_{0}(\mathcal{M}^{L}) \longrightarrow K_{0}^{\rm sp}(\mathcal{M})$ is inverse to the natural inclusion $ K_{0}^{\rm sp}(\mathcal{M})\hookrightarrow K_{0}(\mathcal{M}^{L})$. Similarly, we can prove that $\ind^{R}_{\mathcal{M}}:	K_{0}(\mathcal{M}^{R}) \longrightarrow K_{0}^{\rm sp}(\mathcal{M})$ is an isomorphism. 
		\end{proof}
		
		\begin{remark}\label{R-3.16} Let $\mathcal{T}$ be a triangulated category with the suspension functor $\Sigma$. Following \cite[Section 4]{Ch}, a full additive subcategory  $\mathcal{M}$ of $\mathcal{T}$ is $d$-rigid if $\Hom(\mathcal{M},\Sigma^{i}\mathcal{M})=0$ for $1\leq i<d$. In this case, $\mathcal{M}$ is actually a $(d-1)$-rigid subcategory in the sense of Definition \ref{D-3.1}(1).

			(1) Set $$\mathcal{F}_d=\Sigma^{-(d-1)}\mathcal{M}\ast\cdots\ast\Sigma^{-1}\mathcal{M}\ast\mathcal{M}.$$
			Observe that $\mathcal{F}_d=\mathcal{M}^{d-1}$. Thus Theorem \ref{main}(1) generalizes \cite[Proposition 4.5]{Ch}. 
			
			(2) Let $X$ be an object in $\mathcal{T}$. A $ \Sigma^{\leq0}(\mathcal{M})$-{\em filtration of length} $n\geq 1$ (cf. \cite[Definition 2.2]{Ch}) for $X$ means a sequence of morphisms 
			\begin{equation*}
				0=X_{n}{\longrightarrow}X_{n-1}\stackrel{}{\longrightarrow}\cdots \stackrel{}{\longrightarrow}X_{1}\stackrel{}{\longrightarrow} X_0=X
			\end{equation*}
			such that each morphism fits into a triangle
			\begin{equation*}
				X_{i+1}\stackrel{}{\longrightarrow}X_{i}\stackrel{}{\longrightarrow}\Sigma^{-i}(M_i^{X})\rightarrow \Sigma X_{i+1}
			\end{equation*}
			with $M_i^{X}\in M$ for $0\leq i\leq n-1$. We define an element
			$$\gamma(X)=\sum_{i=0}^{n-1}(-1)^{i}[M_i^{X}]\in K_{0}^{\rm sp}(\mathcal{M}).$$
			Indeed,  there exists an $\mathbb{E}$-triangle sequence
			$$(X\stackrel{}{\longrightarrow}M_0^{X}\stackrel{}{\longrightarrow}\cdots \stackrel{}{\longrightarrow}M_{n-2}^{X}\stackrel{}{\longrightarrow}M_{n-1}^{X}{\longrightarrow}0)\sim(\Sigma X_1,\Sigma^{2} X_2,\cdots,\Sigma^{n-1} X_n\cong M_{n-1}^{X}).$$
			Thus $X\in\mathcal{M}^{n-1}$ and $\gamma(X)=\ind^{ L}_{\mathcal{M}}(X)$. 
			
			Now we assume that $\mathcal{F}_d$ is closed under extensions. Denote by $N$ the  subgroup of $K_{0}^{\rm sp}(\mathcal{M})$ generated by the elements
			$$\gamma(X)-[M_1]-(-1)^{d-1}[M_2]$$
			for all triangles
			\begin{equation*}
				M_1\stackrel{}{\longrightarrow}X\stackrel{}{\longrightarrow}\Sigma^{-(d-1)}(M_2)\rightarrow \Sigma M_{1}
			\end{equation*}
			with $M_1,M_2\in\mathcal{M}$. Note that $X\in \mathcal{M}\ast\Sigma^{-(d-1)}\mathcal{M}\subseteq\mathcal{F}_d$ (cf. \cite[Lemma 4.3]{Ch}). 	It was shown in \cite[Theorem 4.8]{Ch} that the inclusion $ \mathcal{M}\hookrightarrow\mathcal{F}_d$ induces a group isomorphism $i:K_{0}^{\rm sp}(\mathcal{M})/N\rightarrow K_{0}(\mathcal{F}^{d})$  whose inverse sends $[X]$ to $\overline{\ind^{ L}_{\mathcal{M}}}(X):=[\gamma(X)]=[\ind^{ L}_{\mathcal{M}}(X)]$. On the other hand, we have group isomorphism $K_{0}(\mathcal{M}^{d-1})\cong K_{0}^{\rm in}(\mathcal{M})$ by Theorem \ref{main}(2). We emphasize that this isomorphism does not require $\mathcal{F}_d$ to be extension-closed. By combining \cite[Theorem 4.8]{Ch} and Theorem \ref{main}(2), we have group isomorphisms
			\[
			\begin{tikzcd}[column sep =40.5, row sep =40.5]
				K_{0}^{\rm sp}(\mathcal{M})/N  	\rar[leftarrow,"\overline{\ind^{ L}_{\mathcal{M}}}","\cong"'] &
				K_{0}(\mathcal{F}_{d})=K_{0}(\mathcal{M}^{d-1})\rar[rightarrow,"\ind^{L}_{\mathcal{M}}","\cong"']
				&K_{0}^{\rm in}(\mathcal{M})=K_{0}^{\rm sp}(\mathcal{M})/\im\Phi. 
			\end{tikzcd}
			\]
		\end{remark}

		\section{Grothendieck groups and silting subcategories}
		Our aim in this section is to investigate the Grothendieck group  $K_{0}(\mathscr{C})$ via   silting subcategories. Then we establish several isomorphisms of $K_{0}(\mathscr{C})$ in terms of silting subcategories and bounded hereditary  cotorsion pairs.
		
		We	say a subcategory  $\mathcal{M}\subseteq\mathscr{C}$  is {\em thick} if $\mathcal{M}$ is closed under extensions, cones, cocones and direct summands. We denote by $\thick(\mathcal{M})$ the  smallest thick subcategory containing $\mathcal{M}$. 
		
		\begin{definition}\label{D-4.1} A $\infty$-rigid subcategory $\mathcal{M}$ is  {\em silting} if $\mathcal{M}=\add(\mathcal{M})$ and $\thick(\mathcal{M})=\mathscr{C}$.
		\end{definition}
		
		Let $(\mathcal{T},\mathcal{F})$ be a cotorsion pair in $\mathscr{C}$.  We say  $(\mathcal{T},\mathcal{F})$ is  {\em bounded} if $\mathcal{T}^{R}=\mathcal{F}^{L}=\mathscr{C}$. It was proved in \cite[Theorem 5.7]{Ad} that there exists a bijection between silting subcategories and bounded hereditary  cotorsion pairs. The bijection sends a silting subcategory $\mathcal{M}$ to the bounded hereditary  cotorsion pair $(\mathcal{M}^{L},\mathcal{M}^{R})$. Its inverse is given by $(\mathcal{T},\mathcal{F})\mapsto\mathcal{T}\cap\mathcal{F} $. By this, we introduce the following notation.

		\begin{definition}\label{D-4-2} Let $\mathcal{M}$ be a silting subcategory of $\mathscr{C}$. For any $X\in\mathscr{C}$, there exist two $\mathbb{E}$-triangles
			$$F_{X}\stackrel{}{\longrightarrow} T_{X}{\longrightarrow}X\stackrel{}\dashrightarrow~\text{and}~X\stackrel{}{\longrightarrow} F^{X}{\longrightarrow}T^{X}\stackrel{}\dashrightarrow$$
			such that $F_{X},F^{X}\in\mathcal{M}^{R}$ and $T_{X},T^{X}\in\mathcal{M}^{L}$. The {\em left silting index} of $X$ with respect to $\mathcal{M}$ is the element
			$$\mathbb{L}(X)=\ind^{R}_{\mathcal{M}}(F^{X})-\ind^{L}_{\mathcal{M}}(T^{X})\in K_{0}^{\rm sp}(\mathcal{M}).$$
			Dually, the {\em right silting index} of $X$ with respect to $\mathcal{M}$ is the element
			$$\mathbb{R}(X)=\ind^{L}_{\mathcal{M}}(T_{X})-\ind^{R}_{\mathcal{M}}(F_{X})\in K_{0}^{\rm sp}(\mathcal{M}).$$
		\end{definition}
		
		In this section, we fix a  silting subcategory  $\mathcal{M}$ of $\mathscr{C}$. Since $\mathcal{M}$ is $\infty$-rigid, the definitions of left index and right index are well-defined (see Remark \ref{R-2-6}). The following characterization of silting index will be useful.

		\begin{lemma}\label{L-3-1} Take $X\in\mathscr{C}$.
			
			$(1)$ If $X\in\mathcal{M}^{L}$, then $\mathbb{R}(X)=\ind^{L}_{\mathcal{M}}(X)$.
			
			$(2)$ If $X\in\mathcal{M}^{R}$, then $\mathbb{R}(X)=\ind^{R}_{\mathcal{M}}(X)$.
			
			$(3)$ Suppose that $X\cong T\oplus F$ with $T\in\mathcal{M}^{L}$ and  $F\in\mathcal{M}^{R}$. Then
			$$\mathbb{R}(X)=\ind^{L}_{\mathcal{M}}(T)+\ind^{R}_{\mathcal{M}}(F).$$
			
			$(4)$ The element   $\mathbb{R}(X)$  does not depend on the choice of $\mathbb{E}$-triangle.
			
		\end{lemma}
		\begin{proof}We take two $\mathbb{E}$-triangles
			$$F_{1}\stackrel{}{\longrightarrow} T_{1}{\longrightarrow}X\stackrel{}\dashrightarrow~\text{and}~F_{2}\stackrel{}{\longrightarrow} T_{2}{\longrightarrow}X\stackrel{}\dashrightarrow$$
			with $T_{1},T_{2}\in\mathcal{M}^{L}$ and $F_{1},F_{2}\in\mathcal{M}^{R}$.
			
			(1) Since $X\in\mathcal{M}^{L}$, we have $T_{1}\cong F_{1}\oplus X$ and hence $F_{1}\in\mathcal{M}^{L}\cap\mathcal{M}^{R}=\mathcal{M}$. By using  Proposition \ref{L-7}(1), we infer that 
			$$\mathbb{R}(X)=\ind^{L}_{\mathcal{M}}(T_{1})-\ind^{R}_{\mathcal{M}}(F_{1})=\ind^{L}_{\mathcal{M}}(T_{1})-[F_{1}]=\ind^{L}_{\mathcal{M}}(X).$$

			$(2)$ Since $X\in\mathcal{M}^{R}$, we have $T_{1}\in\mathcal{M}^{L}\cap\mathcal{M}^{R}=\mathcal{M}$. Then Proposition \ref{L-7}(2) implies that 
			$$\mathbb{R}(X)=\ind^{L}_{\mathcal{M}}(T_{1})-\ind^{R}_{\mathcal{M}}(F_{1})=[T_{1}]-\ind^{R}_{\mathcal{M}}(F_{1})=\ind^{R}_{\mathcal{M}}(X).$$

			(3)  Applying $\rm (ET4)^{op}$, we have  the following commutative diagram
			\begin{equation*}
				\xymatrix{
					F_{1}\ar[r]^{} \ar@{=}[d]_{}& K \ar[r]^{}\ar[d] &  F\ar@{-->}[r]^{} \ar[d]_{} &  \\
					F_{1}\ar[r]^{} & T_{1} \ar[r]^{}\ar[d] &  X\ar@{-->}[r]^{}\ar[d] &  \\
					&  T \ar@{=}[r] \ar@{-->}[d]^{0}& T\ar@{-->}[d]^{0}&\\
					&  &  &  }
			\end{equation*}
			Since $F_{1},F\in\mathcal{M}^{R}$, we have $K\in\mathcal{M}^{R}$ and hence $T_{1}\cong K\oplus T$. In particular, we have $K\in\mathcal{M}^{L}\cap\mathcal{M}^{R}=\mathcal{M}$. By (1), we have
			$$\mathbb{R}(T)=\ind^{L}_{\mathcal{M}}(T)=\ind^{L}_{\mathcal{M}}(T_{1})-\ind^{R}_{\mathcal{M}}(K)=\ind^{L}_{\mathcal{M}}(T_{1})-[K].$$
			By using Proposition \ref{L-7}(2), we obtain that
			$$[K]=\ind^{R}_{\mathcal{M}}(K)=\ind^{R}_{\mathcal{M}}(F_{1})+\ind^{R}_{\mathcal{M}}(F).$$
			and thus
			$$\mathbb{R}(X)=\ind^{L}_{\mathcal{M}}(T_{1})-\ind^{R}_{\mathcal{M}}(F_{1})=\ind^{L}_{\mathcal{M}}(T)+\ind^{R}_{\mathcal{M}}(F).$$
			
			$(4)$
			Since $\mathbb{E}(\mathcal{M}^{L},\mathcal{M}^{R})=0$, we have $T_{1}\oplus F_{2}\cong F_{1}\oplus T_{2}$. By (3), we get
			$$\ind^{L}_{\mathcal{M}}(T_{1})+\ind^{R}_{\mathcal{M}}(F_{2})=\ind^{L}_{\mathcal{M}}(T_{2})+\ind^{R}_{\mathcal{M}}(F_{1})$$
			and thus
			$$\ind^{L}_{\mathcal{M}}(T_{1})-\ind^{R}_{\mathcal{M}}(F_{1})=\ind^{L}_{\mathcal{M}}(T_{2})-\ind^{R}_{\mathcal{M}}(F_{2}).$$
			This shows that  $\mathbb{R}(X)$ does not depend on the choice of $\mathbb{E}$-triangle.
		\end{proof}
		
		It is clear that $\mathbb{R}(X\oplus Y)=\mathbb{R}(X)\oplus \mathbb{R}(Y)$ for any $X,Y\in\mathscr{C}$. Thus Lemma \ref{L-3-1} implies that there is a map $\mathbb{R}:K_{0}^{\rm sp}(\mathscr{C})\rightarrow K_{0}^{\rm sp}(\mathcal{M}); [X]\mapsto\mathbb{R}(X)$. 
		
		\begin{proposition}\label{P-3-1} The map $\mathbb{R}$ descends to the  Grothendieck group $K_{0}(\mathscr{C})$.
		\end{proposition}
		\begin{proof} Let $A\stackrel{}{\longrightarrow}B\stackrel{}{\longrightarrow}C\stackrel{}\dashrightarrow$ be an $\mathbb{E}$-triangle in $\mathscr{C}$. It suffices to show that $\mathbb{R}(B)=\mathbb{R}(A)+\mathbb{R}(C)$. Since $(\mathcal{M}^{L},\mathcal{M}^{R})$ is a cotorsion pair, there exists an $\mathbb{E}$-triangle
			$$ F\stackrel{}{\longrightarrow}T\stackrel{}{\longrightarrow}B\stackrel{}\dashrightarrow$$
			such that $T\in\mathcal{M}^{L}$ and $F\in\mathcal{M}^{R}$. Then ${\rm (ET4)^{op}}$ yields the following commutative diagram
			\begin{equation*}
				\xymatrix{
					F\ar[r]^{} \ar@{=}[d]_{}& K \ar[r]^{}\ar[d] &  A\ar@{-->}[r]^{} \ar[d]_{} &  \\
					F\ar[r]^{} & T \ar[r]^{}\ar[d] &  B\ar@{-->}[r]^{}\ar[d] &  \\
					&  C \ar@{=}[r] \ar@{-->}[d]^{}& C\ar@{-->}[d]^{}&\\
					&  &  &  }
			\end{equation*}
			
			$\mathbf{Claim~1}.$ If $C\in\mathcal{M}^{L}$, then $\mathbb{R}(B)=\mathbb{R}(A)+\mathbb{R}(C)$.
			
			Since $\mathcal{M}^{L}$ is closed under cocones, we have $K\in\mathcal{M}^{L}$ and thus $\mathbb{R}(A)=\ind^{L}_{\mathcal{M}}(K)-\ind^{R}_{\mathcal{M}}(F)$. Using Proposition \ref{L-7}(1) together with  Lemma \ref{L-3-1}(1), we have
			$$\mathbb{R}(C)=\ind^{L}_{\mathcal{M}}(C)=\ind^{L}_{\mathcal{M}}(T)-\ind^{L}_{\mathcal{M}}(K).$$
			The observation above implies that
			$$\mathbb{R}(B)=\ind^{L}_{\mathcal{M}}(T)-\ind^{R}_{\mathcal{M}}(F)=(\mathbb{R}(C)+\ind^{L}_{\mathcal{M}}(K))-(\ind^{L}_{\mathcal{M}}(K)-\mathbb{R}(A))=\mathbb{R}(A)+\mathbb{R}(C).$$
			
			$\mathbf{Claim~2}.$ If $A\in\mathcal{M}^{R}$, then $\mathbb{R}(B)=\mathbb{R}(A)+\mathbb{R}(C)$.
			
			Since $\mathcal{M}^{R}$ is closed under extensions, we have $K\in\mathcal{M}^{R}$ and thus $\mathbb{R}(C)=\ind^{L}_{\mathcal{M}}(T)-\ind^{R}_{\mathcal{M}}(K)$. Using Proposition \ref{L-7}(2) together with  Lemma \ref{L-3-1}(2), we have
			$$\mathbb{R}(A)=\ind^{R}_{\mathcal{M}}(A)=\ind^{R}_{\mathcal{M}}(K)-\ind^{R}_{\mathcal{M}}(F).$$
			Then we get
			$$\mathbb{R}(B)=\ind^{L}_{\mathcal{M}}(T)-\ind^{R}_{\mathcal{M}}(F)=(\mathbb{R}(C)+\ind^{R}_{\mathcal{M}}(K))-(\ind^{R}_{\mathcal{M}}(K)-\mathbb{R}(A))=\mathbb{R}(A)+\mathbb{R}(C).$$
			
			Now, we consider the general case. We take an $\mathbb{E}$-triangle
			$$ F'\stackrel{}{\longrightarrow}T'\stackrel{}{\longrightarrow}C\stackrel{}\dashrightarrow$$
			such that $T'\in\mathcal{M}^{L}$ and $F'\in\mathcal{M}^{R}$. Consider the following commutative diagram
			\begin{equation*}
				\xymatrix{
					&F'\ar@{=}[r]^{}\ar[d] &  F' \ar[d]_{} &  \\
					A\ar[r]^{} \ar@{=}[d]^{}& Q \ar[r]^{}\ar[d] &  T'\ar@{-->}[r]^{}\ar[d] &  \\
					A\ar[r]^{} &  B \ar[r] \ar@{-->}[d]^{}& C\ar@{-->}[d]^{}\ar@{-->}[r]^{}&\\
					&  &  &  }
			\end{equation*}
			Then we have
			\begin{align*}
				\mathbb{R}(B)&=\mathbb{R}(Q)-\mathbb{R}(F')=\mathbb{R}(Q)-\ind^{R}_{\mathcal{M}}(F')~~~~&\text{(By~Claim 2 and Lemma \ref{L-3-1}(2))}\\
				&=\ind^{L}_{\mathcal{M}}(T')+\mathbb{R}(A)-\ind^{R}_{\mathcal{M}}(F')~~~~&\text{(By Claim 1 and Lemma \ref{L-3-1}(1))}\\
				&=\mathbb{R}(A)+\mathbb{R}(C).
			\end{align*}
			This completes the proof.
		\end{proof}
		
		Now, we can state the main result of this section. 
		
		\begin{theorem}\label{main2}
			Let $\mathcal{M}$ be a silting subcategory of  $\mathscr{C}$. There  are inverse group isomorphisms
			\[
			\begin{tikzcd}[column sep =40.5, row sep =40.5]
				K_{0}(\mathscr{C}) \rar[rightarrow,"\mathbb{L}","\cong"']  &   K_{0}^{\rm sp}(\mathcal{M})&K_{0}(\mathscr{C}). \lar[rightarrow,"\mathbb{R}"',"\cong"]
			\end{tikzcd}
			\]
			In particular, if  $\mathscr{C}$ is Krull–Schmidt, then all silting subcategories  have the same number of indecomposable objects up to isomorphism.
		\end{theorem}
		\begin{proof} By proposition \ref{P-3-1}, the homomorphism $\mathbb{R}:K_{0}(\mathscr{C})\longrightarrow K_{0}^{\rm sp}(\mathcal{M})$
			is well-defined. Let $i: K_{0}^{\rm sp}(\mathcal{M})\rightarrow K_{0}(\mathscr{C})$ be the natural projection functor. For any $X\in\mathscr{C}$, we take an $\mathbb{E}$-triangle $F\stackrel{}{\longrightarrow} T\stackrel{}{\longrightarrow}X\stackrel{}\dashrightarrow$ with  $T\in\mathcal{M}^{L}$ and $F\in\mathcal{M}^{R}$. Recall that $\ind^{L}_{\mathcal{M}}(T)=[T]~\text{and}~\ind^{R}_{\mathcal{M}}(F)=[F]$ in $K_{0}(\mathscr{C})$. It means that 
			$$\mathbb{R}(X)=\ind^{L}_{\mathcal{M}}(T)-\ind^{R}_{\mathcal{M}}(F)=[T]-[F]~\text{in}~ K_{0}^{\rm sp}(\mathcal{M}).$$
			Thus $i\mathbb{R}(X)=[T]-[F]=[X]\in	K_{0}(\mathscr{C})$. Conversely, we have $\mathbb{R}i(M)=[M]$ for any $M\in\mathcal{M}$. A similar argument shows that $\mathbb{L}:K_{0}(\mathscr{C})\longrightarrow K_{0}^{\rm sp}(\mathcal{M})$ is an isomorphism.
		\end{proof}

		\begin{remark} \label{R-4-6}
			$(1)$ By Theorem \ref{main} and Theorem \ref{main2}, we have the following commutative diagram and the homomorphisms are bijections.
			
			\[
			\begin{tikzcd}[column sep =40.5, row sep =40.5]
				K_{0}(\mathscr{C})   \dar["\mathbb{R}",shift left] \rar["\mathbb{R}" ,shift left]  &  K_{0}^{\rm sp}(\mathcal{M}) \dar[twoheadrightarrow,"\pi_{3}",shift left] \rar[twoheadrightarrow, "\pi_{2}"',shift right]
				\lar[twoheadrightarrow, "\pi_{1}",shift left]&	K_{0}(\mathscr{C})\lar["\mathbb{L}"' ,shift right] \\
				K_{0}(\mathcal{M}^{L})\rar["{\rm ind}^{L}_{\mathcal{M}}",shift left]\uar[hookrightarrow,"i",shift left] &     K_{0}(\mathcal{M}^{R})\lar["{\rm ind}^{R}_{\mathcal{M}}",shift left]\uar["{\rm ind}^{R}_{\mathcal{M}}",shift left]&
			\end{tikzcd}
			\]
			where $i$ is an inclusion and $\pi_{1},\pi_{2},\pi_{3}$ are natural projections. 
			
			$(2)$ The isomorphism  between $K_{0}(\mathscr{C})$ and $K_{0}^{\rm sp}(\mathcal{M})$ is well-known and it was proved by \cite[Theorem 5.3.1]{Bo}, \cite[Theorem 2.27]{Ai} and \cite[Theorem 3.3]{Ch} in triangulated categories.  This result was also proved by  \cite[Theorem 4.1]{Ad2} in  Krull-Schmidt extriangulated categories. It should be noted that \cite[Theorem 3.3]{Ch} does not require the silting subcategory to be closed under direct summands, compare Definition \ref{D-4.1}.

		\end{remark}
		
		Recall that there is a bijection between silting subcategories and bounded hereditary  cotorsion pairs.  By this, we can give another version of the Theorem \ref{main2}.

		\begin{corollary}\label{C-5-7} Let $(\mathcal{T},\mathcal{F})$ be a bounded hereditary  cotorsion pair in $\mathscr{C}$. Then
			$$K_{0}(\mathcal{T})\cong K_{0}(\mathcal{F})\cong K_{0}(\mathscr{C}) \cong K_{0}^{\rm sp}(\mathcal{T}\cap\mathcal{F}).$$
		\end{corollary}
		\begin{proof} This follows from \cite[Theorem 5.7]{Ad} and Remark \ref{R-4-6}(1).
		\end{proof}
		
		\begin{remark}\label{R-4-8}  We note that a hereditary cotorsion pair is  precisely  a {\em co-t-structure} (also known as {\em weight structure}) in triangulated categories (cf. \cite[Example 4.1]{Ad}). Then Corollary \ref{C-5-7} recovers \cite[Theorem 5.3.1]{Bo} (see also \cite[Corollary 3.4]{Ch}).
		\end{remark}

		\section{Grothendieck groups and  higher cluster tilting subcategories}
		In this section,  we investigate the Grothendieck group $K_{0}(\mathscr{C})$ via a higher cluster subcategory $\mathcal{T}$ in $\mathscr{C}$. We establish several  group isomorphisms concerning $K_{0}(\mathscr{C})$ and $K_{0}^{\rm sp}(\mathcal{T})$. This generalizes and strengthens the corresponding results in various settings (see Remark \ref{R-5-5}).

		We say a functorially finite subcategory $\mathcal{T}$ is {\em strong}  (cf. \cite[Definition 3.19]{Zh}) if for any $X\in\mathscr{C}$, there exist two $\mathbb{E}$-triangles
		$$ X\stackrel{f}{\longrightarrow}T\stackrel{}{\longrightarrow}T_{1}\stackrel{}\dashrightarrow~\text{and}~U\stackrel{}{\longrightarrow}U_{1}\stackrel{g}{\longrightarrow}X\stackrel{}\dashrightarrow$$
		such that $f$ is a left $\mathcal{T}$-approximation and $g$ is a right $\mathcal{T}$-approximation. When $\mathscr{C}$ has enough projective objects and injective objects, one can check that functorially finite subcategories are  strong.
		\begin{definition}\label{D-4-9}A subcategory $\mathcal{T}$ of $\mathscr{C}$ is called $d$-{\em cluster tilting} if
			$\mathcal{T}$ is strong functorially finite and $\mathcal{T}=\mathcal{T}^{\perp_{[1,d-1]}}={^{\perp_{[1,d-1]}}}\mathcal{T}$. We say an object $T\in\mathscr{C}$ is {\em $d$-cluster tilting} if $\add T$ is a $d$-cluster tilting subcategory.
		\end{definition}
		In this section, we always assume that   $\mathcal{T}$ is a  $(d+1)$-cluster tilting subcategories of $\mathscr{C}$. It is obvious that $\mathcal{T}$ is $d$-rigid. The following result is analogous to \cite[Lemma 3.4]{Wl}.
		\begin{lemma}\label{L-3-10} We have $\mathscr{C}=\mathcal{T}^{d}=\mathcal{T}_{d}.$
		\end{lemma}
		\begin{proof} Since $\mathcal{T}$ is strong functorially finite, each object $X$ in $\mathscr{C}$ admits an $\mathbb{E}$-triangle
			$$ X\stackrel{f}{\longrightarrow}T_{0}\stackrel{}{\longrightarrow}K_{1}\stackrel{}\dashrightarrow$$
			such that $f$ is a left $\mathcal{T}$-approximation. It is straightforward to check that $K_{1}\in{^{\perp_{1}}}\mathcal{T}$ and $\mathbb{E}^{i+1}(K_1,\mathcal{T})\cong\mathbb{E}^{i}(X,\mathcal{T})$ for any $1\leq i\leq d-1$. We now repeat the argument to the object $K_{1}$. Then we obtain an $\mathbb{E}$-triangle
			sequence \begin{equation*}
				X\stackrel{}{\longrightarrow}T_{0}\stackrel{}{\longrightarrow}\cdots \stackrel{}{\longrightarrow}T_{d-1}\stackrel{}{\longrightarrow}T_{d}
			\end{equation*}
			such that $T_{i}\in \mathcal{T}$ for $1\leq i\leq d-1$. In particular, we have $T_{d}\in\mathcal{T}^{\perp_{[1,d]}}=\mathcal{T}$. This shows that $\mathscr{C}=\mathcal{M}^{d}$. A similar argument shows that $\mathscr{C}=\mathcal{M}_{d}$.
		\end{proof}
		\begin{remark}
			By	Lemma \ref{L-3-10}, we obtain that $\mathscr{C}=\mathcal{T}^{d}=\mathcal{T}_{d}=\mathcal{T}^{L}=\mathcal{T}^{R}.$
		\end{remark}
		
		Now, we consider two  subfunctors $\mathbb{E}_{\mathcal{T}}$ and  $\mathbb{E}^{\mathcal{T}}$ defined as follows:
		$$\mathbb{E}_{\mathcal{T}}(C,A)=\{\delta\in\mathbb{E}(C,A)~|~(\delta_{\sharp})_{T}=0~\text{for~any}~T\in\mathcal{T}\},$$
		$$\mathbb{E}^{\mathcal{T}}(C,A)=\{\delta\in\mathbb{E}(C,A)~|~(\delta^{\sharp})_{T}=0~\text{for~any}~T\in\mathcal{T}\}.$$
		These functors induces two  {\em relative} extriangulated categories $(\mathscr{C},\mathbb{E}_{\mathcal{T}},\mathfrak{s}|_{\mathbb{E}_{\mathcal{T}}})$ and  $(\mathscr{C},\mathbb{E}^{\mathcal{T}},\mathfrak{s}|_{\mathbb{E}^{\mathcal{T}}})$ (c.f. \cite[Proposition 3.17]{He}).
		The Grothendieck group $K_{0}(\mathscr{C},\mathbb{E}_{\mathcal{T}},\mathfrak{s}|_{\mathbb{E}_{\mathcal{T}}})$  is defined as
		$$K_{0}(\mathscr{C},\mathbb{E}_{\mathcal{T}},\mathfrak{s}|_{\mathbb{E}_{\mathcal{T}}}):=K_{0}^{\rm sp}(\mathscr{C})/\lr{[A]-[B]+[C]~|~A\rightarrow B\rightarrow C\dashrightarrow\text{is an $\mathbb{E}_{\mathcal{T}}$-triangle}}.$$
		Similarly, one can define  $K_{0}(\mathscr{C},\mathbb{E}^{\mathcal{T}},\mathfrak{s}|_{\mathbb{E}^{\mathcal{T}}})$.

		Now we can state the main result of this section.
		
		\begin{theorem}\label{main3}  Let $\mathcal{T}$ and $\mathcal{U}$ be two $(d+1)$-cluster tilting subcategories of $\mathscr{C}$.
			
			$(1)$ If $d=1$,  there are inverse group isomorphisms
			\[
			\begin{tikzcd}[column sep =40.5, row sep =40.5]
				{K_{0}^{\rm sp}(\mathcal{T})}  \rar["{\rm ind}^{ R}_{\mathcal{U}}" ,shift left]  & {K_{0}^{\rm sp}(\mathcal{U})}.
				\lar[ "{\rm ind}^{ L}_{\mathcal{T}}",shift left]
			\end{tikzcd}
			\] 	
			
			$(2)$ If $1\leq t<d$,  then we have group isomorphisms
			\[
			\begin{tikzcd}[column sep =40.5, row sep =40.5]
				K_{0}(\mathcal{T}^{t})  \rar[rightarrow,"\ind^{L}_{\mathcal{T}}","\cong"']  &   K_{0}^{\rm sp}(\mathcal{T})&K_{0}(\mathcal{T}_{t}). \lar[rightarrow,"\ind^{R}_{\mathcal{T}}"',"\cong"]
			\end{tikzcd}
			\]		
			
			$(3)$ We have group isomorphisms
			\[
			\begin{tikzcd}[column sep =40.5, row sep =40.5]
				K_{0}(\mathscr{C})  \rar[rightarrow,"\ind^{L}_{\mathcal{T}}","\cong"']  &   K_{0}^{\rm in}(\mathcal{T})
				&  	K_{0}(\mathscr{C})  \lar[rightarrow,"\ind^{R}_{\mathcal{T}}"',"\cong"].
			\end{tikzcd}
			\]

			$(4)$ There  are inverse group isomorphisms
			\[
			\begin{tikzcd}[column sep =40.5, row sep =40.5]
				K_{0}(\mathscr{C},\mathbb{E}^{\mathcal{T}},\mathfrak{s}|_{\mathbb{E}^{\mathcal{T}}})
				\rar["{\rm ind}^{L}_{\mathcal{T}}" ,shift left]  & {K_{0}^{\rm sp}(\mathcal{T})}
				\lar["i_1",shift left]\rar["i_2"',shift right]
				&   	K_{0}(\mathscr{C},\mathbb{E}_{\mathcal{T}},\mathfrak{s}|_{\mathbb{E}_{\mathcal{T}}})\lar["{\rm ind}^{ R}_{\mathcal{T}}"',shift right]	
			\end{tikzcd}
			\]
			where $i_1,i_2$ are natural inclusions.

		\end{theorem}
		\begin{proof} (1) This proof is analogous to \cite[Theorem 3.13(1)]{Wl}. We give a sketch here for the convenience of the reader.
			
			For any $[T]\in\mathcal{T}$, there exists an $\mathbb{E}$-triangle 	$ U_1\stackrel{}{\longrightarrow}U_{0}\stackrel{}{\longrightarrow}T\stackrel{}\dashrightarrow$
			such that each $U_i\in\mathcal{U}$. By Lemma \ref{L-7}(1), we have 
			$\ind^{L}_{\mathcal{T}}(U_0)=\ind^{L}_{\mathcal{T}}(U_1)+[T]$. It follows that
			$$\ind^{L}_{\mathcal{T}}(\ind^{R}_{\mathcal{U}}([T]))=\ind^{L}_{\mathcal{T}}([U_0]-[U_1])=[T].$$
			Dually, we have $\ind^{R}_{\mathcal{U}}(\ind^{L}_{\mathcal{T}}([U]))=[U]$ for any 
			$[U]\in\mathcal{U}$.
			
			$(2)$ It follows from Theorem \ref{main}(1).
			
			$(3)$ It follows from    Theorem \ref{main}(2) and Lemma \ref{L-3-10}.

			$(4)$  Let 	$ A\stackrel{}{\longrightarrow}B\stackrel{}{\longrightarrow}C\stackrel{}\dashrightarrow$ be an $\mathbb{E}^{\mathcal{T}}$-triangle. Then Lemma \ref{L-7}(1) implies that $\ind^{ L}_{\mathcal{T}}(A)-\ind^{ L}_{\mathcal{T}}(B)+\ind^{ L}_{\mathcal{T}}(C)=0$. This gives a well-defined homomorphism
			$${\rm ind}^{ L}_{\mathcal{T}}:K_{0}(\mathscr{C},\mathbb{E}^{\mathcal{T}},\mathfrak{s}|_{\mathbb{E}^{\mathcal{T}}})\rightarrow{K_{0}^{\rm sp}(\mathcal{T})}.$$
			It  remains to check that  ${\rm ind}^{\rm L}_{\mathcal{T}}$ and $i_1$  are mutually inverse to each other. It is obvious that ${\rm ind}^{ L}_{\mathcal{T}}(i_1([T]))=[T]$ for any $T\in\mathcal{T}$. Take $X\in\mathscr{C}$. By Lemma \ref{L-3-10}, there exists an $\mathbb{E}$-triangle sequence
			\begin{equation*}
				\eta:	(X\stackrel{}{\longrightarrow}X_{0}\stackrel{}{\longrightarrow}\cdots \stackrel{}{\longrightarrow}X_{t-1}\stackrel{}{\longrightarrow}X_{t})\sim(K_{0},K_{1},\cdots,K_{t-2})
			\end{equation*}
			such that each $X_i\in\mathcal{T}$. Note that $\mathbb{E}(K_i,\mathcal{T})=0$ by Lemma \ref{L-A}(1). Thus $\eta$ is also an  $\mathbb{E}^{\mathcal{T}}$-triangle sequence in $K_{0}(\mathscr{C},\mathbb{E}^{\mathcal{T}},\mathfrak{s}|_{\mathbb{E}^{\mathcal{T}}})$. The observation above implies that $i_2({\rm ind}^{ L}_{\mathcal{T}}(X))=[X]$. A similar argument shows that
			$\ind^{R}_{\mathcal{T}}$ and $i_2$  are mutually inverse to each other.
		\end{proof}
		We make some comments on how our work generalizes some known results.
		
		\begin{remark}\label{R-5-5} $(1)$ The assertions (1) and (3) in  Theorem \ref{main3} are analogous to \cite[Theorem 3.13]{Wl}.  By contrast,  we do not require that $\mathscr{C}$ be Krull-Schmidt and satisfy the condition (WIC). Theorem \ref{main3}(2) recovers \cite[Proposition 4.5]{Ch}. Theorem \ref{main3}(4) recovers \cite[Theorem 4.10]{Jo}. 
			
			$(2)$ Let $\mathscr{C}$ be a triangulated category  with the suspension functor $\Sigma$. Let $\mathcal{T}$ be a $d$-cluster tilting subcategory of $\mathscr{C}$. By combining \cite[Corollary 4.10]{Ch} and  Theorem \ref{main3}(3), we have group isomorphisms
			$$K_{0}^{\rm sp}(\mathcal{T})/\im\Phi=K_{0}^{\rm in}(\mathcal{T})\cong	K_{0}(\mathscr{C})\cong K_{0}^{\rm sp}(\mathcal{T})/N$$
			We refer the reader to \cite[Section 4]{Ch} (or Remark \ref{R-3.16}(2)) for the explicit definition of $N$.
			
			Now we assume that  $\Sigma^{d}\mathcal{T}=\mathcal{T}$. Then $(\mathcal{T},\Sigma^{d},\pentagon)$ is a $(d+2)$-angulated category in the sense of \cite{Ge}. Here $\pentagon$ is the class of all $(d+2)$-angulated sequences
			\begin{equation*}
				X_{1}\stackrel{}{\longrightarrow}X_{2}\stackrel{}{\longrightarrow}\cdots \stackrel{}{\longrightarrow}X_{d+1}\stackrel{}{\longrightarrow}X_{d+2}{\longrightarrow}\Sigma^{d}X_{1}
			\end{equation*}
			such that \begin{equation*}
				X_{1}\stackrel{}{\longrightarrow}X_{2}\stackrel{}{\longrightarrow}\cdots \stackrel{}{\longrightarrow}X_{d+1}\stackrel{}{\longrightarrow}X_{d+2}
			\end{equation*}
			is an $\mathbb{E}$-triangle sequence. 	Then Remark \ref{R-3-13}  implies that  $K_{0}^{\rm in}(\mathcal{T})$ is precisely the Grothendieck group of $\mathcal{T}$ as $(d+2)$-angulated category in the sense of \cite[Definition 2.2]{FE}.
			Thus Theorem \ref{main3}(3) recovers \cite[Theorem C]{FE}.
		\end{remark}
		

		\section{Grothendieck groups and $d$-cluster categories }\label{Se4.3} 
		The aim of this section is to compute the Grothendieck groups of the $d$-cluster categories of type $A_n$ by using Theorem \ref{main3}. For this purpose, we review the definition of $d$-cluster categories and their geometric model. We refer to \cite{Th,Ba,Ja} for more details.	
		
		Let $Q$ be an acyclic quiver and let $D^{b}(\mod kQ)$ be the bounded derived category of the category $\mod kQ$ of finitely generated modules. The {\em  $d$-cluster category} $\mathcal{C}_{Q}^{d}$ is defined as the orbit category $$D^{b}(\mod kQ)/\tau^{-1}\Sigma^{d},$$
		where  $\tau$ is the Auslander-Reiten translation and $\Sigma$ is the suspension functor. It is known that  $\mathcal{C}_{Q}^{d}$ is Krull-Schmidt, $(d+1)$-Calabi-Yau, and triangulated.
		
		Let $\mathcal{T}$ be a $d$-cluster tilting subcategory of $\mathcal{C}_{A_n}^{d}$. The notion of  Auslander–Reiten $(d + 3)$-angles  was introduced in \cite[Section 3]{Iy}. Since $\mathcal{C}_{A_n}^{d}$ is $(d+1)$-Calabi-Yau, it has Serre functor $\Sigma^{n}$. By  \cite[Theorem 3.10]{Iy}, any object $X\in\mathcal{T}$ admits an  Auslander–Reiten $(d + 3)$-angle
		$$X{\longrightarrow}T_1{\longrightarrow}\cdots \stackrel{}{\longrightarrow} T_{d+1}\stackrel{}{\longrightarrow}X$$
		with terms in $\mathcal{T}$. Then we have the following observation.
		\begin{proposition}\label{P-6.1} We have 
			
			$$\im\Phi=\lr{[X]+\sum_{i=1}^{d+1}(-1)^{i}[X_{i}]+(-1)^{d}[X]~|~X\in\ind{\mathcal{T}},~X{\longrightarrow}T_1{\longrightarrow}\cdots \stackrel{}{\longrightarrow} T_{d+1}\stackrel{}{\longrightarrow}X$$
				$$\text{is an Auslander–Reiten $(d+3)$-angle  with $T_i\in\mathcal{T}$}}.$$
		\end{proposition}
		\begin{proof}  This follows from \cite[Proposition 2.18, Remark 4.6]{FE} and Proposition \ref{L-3-11}.
		\end{proof}
		
		Let $\Pi$ be a regular
		$(d(n+1)+2)$-gon.  Set $W=(d(n+1)+2)$. We label the vertices of $\Pi$ clockwise by $1,2,\cdots,W$. We regard all operations on vertices of $\Pi$ modulo $W$. A {\em $d$-diagonal} $(a,b)$ is a diagonal of $\Pi$ dividing it into an $(sd+2)$-gon and  $(td+2)$-gon for some $s,t\in \mathbb{Z}_{\geq 1}$. If $a>b$, then a  diagonal $(a,b)$  is a  $d$-diagonal if and only if $(a-b)\equiv 1\mod d$. 
		
		\begin{proposition}\label{P-6-1} {\rm \cite[Proposition 5.4]{Ba}} There exists a bijection between $d$-diagonals in $\Pi$ and indecomposable objects in  $\mathcal{C}_{A_{n}}^{d}$.
		\end{proposition}
		By Proposition \ref{P-6-1}, we can identify each indecomposable object in  $\mathcal{C}_{A_n}^{d}$ with a $d$-diagonal in $\Pi$.  A $(d+2)$-{\em angulation} of $\Pi$ is a maximal set of noncrossing $d$-diagonals. The following result is quite useful.
		
		\begin{proposition}\label{P-6-2} {\rm \cite[Theorem 4.5]{Ja}} There exists a bijection between $(d+1)$-cluster tilting objects in $\mathcal{C}_{A_{n}}^{d}$ and $(d+2)$-{\em angulation} of $\Pi$.
		\end{proposition} 
		
		Let $\mathbb{Z}A_n$ be the stable translation quiver
		associated to $A_n$. As shown in \cite[Proposition 5.5]{Ba}, the Auslander-Reiten quiver of $\mathcal{C}_{A_{n}}^{d}$ is isomorphic to the quotient $\mathbb{Z}A_n/\tau^{-1}\Sigma^{d}$ (See Figures \ref{AR-1} and \ref{AR-2}). The main result of this section is the following.
		
		\begin{figure}[h]
			\centering
			\begin{tikzpicture}[scale=1.5][H]
				\draw[red](1,1.6)--(0,0);
				\draw (0,-0.2)  node{$\scriptstyle (1,d+2)$};
				\draw[very thick](0,1.6)--(8,1.6);
				\draw (0,0)  node{$\scriptstyle \bullet $ };
				\draw (1,1.6)  node{$\scriptstyle \bullet$};
				\draw (1.02,1.75)  node{$\scriptstyle (1,nd+2)$};
				\draw (0,-0.2);
				\draw (1,1.6)--(2.0,0) ;
				\draw (2.0,0)  node{$\scriptstyle \bullet$};
				\draw (1.6,-0.2)  node{$\scriptstyle (1+(n-1)d,nd+2)$};
				\draw (0.95,0.7) node{$\scriptstyle 1$};
				\draw (1.5,1.6)--(2.5,0) ;
				\draw (1.5,1.6)  node{$\scriptstyle \bullet$};
				\draw (2.5,0)  node{$\scriptstyle \bullet$};
				\draw (2.5,0)--(3.5,1.6) ;
				\draw (3.5,1.6)  node{$\scriptstyle \bullet$};
				\draw (2.5,0.7) node{$\scriptstyle 2$};
				\draw (3.5,0.7) node{$\scriptstyle \cdots\cdots$};
				\draw (3.6,0)--(4.6,1.6) ;
				\draw (3.6,0)  node{$\scriptstyle \bullet$};
				\draw (4.6,1.6)  node{$\scriptstyle \bullet$};
				\draw (4.6,1.6)--(5.6,0) ;		
				\draw (5.6,0)  node{$\scriptstyle \bullet$};
				\draw (4.6,0.7) node{$\scriptstyle d$};
				\draw (5.1,1.6)--(6.1,0) ;				
				\draw (5.1,1.6)  node{$\scriptstyle \bullet$};
				\draw (6.1,0) node{$\scriptstyle \bullet$};
				\draw (5.1,1.6)--(6.1,0) ;				
				\draw (5.1,1.6)  node{$\scriptstyle \bullet$};
				\draw (6.1,0) node{$\scriptstyle \bullet$};		
				\draw[very thick](0,0)--(8,0);
				\draw[red] (5.6,1.6)--(6.6,0) ;				
				\draw (5.6,1.6)  node{$\scriptstyle \bullet$};
				\draw (6.6,0) node{$\scriptstyle \bullet$};
				\draw (5.6,1.75)  node{$\scriptstyle (1,d+2)$};
				\draw (6.6,-0.2	)  node{$\scriptstyle (1,nd+2)$};
			\end{tikzpicture}
			\caption{The Auslander quiver of $\mathcal{C}_{A_{n}}^{d}$ when $d$ is odd.}
			\label{AR-1}
		\end{figure}
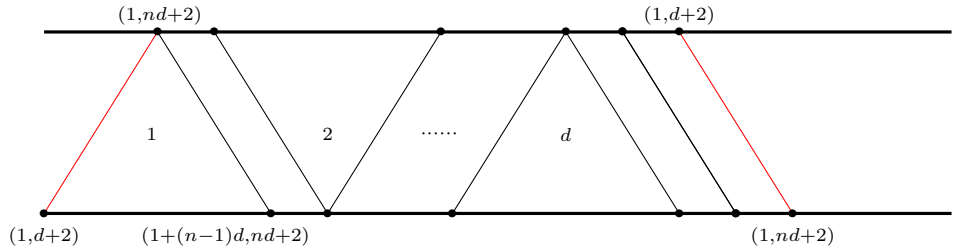
		
		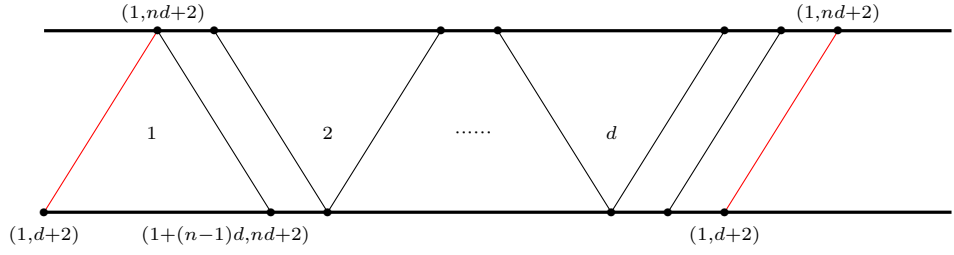
\begin{figure}[h]
			\centering
			\begin{tikzpicture}[scale=1.5][H]
				\draw[red](1,1.6)--(0,0);
				\draw (0,-0.2)  node{$\scriptstyle (1,d+2)$};
				\draw[very thick](0,1.6)--(8,1.6);
				\draw (0,0)  node{$\scriptstyle \bullet $ };
				\draw (1,1.6)  node{$\scriptstyle \bullet$};
				\draw (1.05,1.75)  node{$\scriptstyle (1,nd+2)$};
				\draw (0,-0.2);
				\draw (1,1.6)--(2.0,0) ;
				\draw (2.0,0)  node{$\scriptstyle \bullet$};
				\draw (1.6,-0.2)  node{$\scriptstyle (1+(n-1)d,nd+2)$};
				\draw (0.95,0.7) node{$\scriptstyle 1$};
				\draw (1.5,1.6)--(2.5,0) ;
				\draw (1.5,1.6)  node{$\scriptstyle \bullet$};
				\draw (2.5,0)  node{$\scriptstyle \bullet$};
				\draw (2.5,0)--(3.5,1.6) ;
				\draw (3.5,1.6)  node{$\scriptstyle \bullet$};
				\draw (2.5,0.7) node{$\scriptstyle 2$};
				\draw (3.8,0.7) node{$\scriptstyle \cdots\cdots$};

				\draw (4,1.6)--(5,0) ;
				\draw (4,1.6)  node{$\scriptstyle \bullet$};
				\draw (5,0)  node{$\scriptstyle \bullet$};
				\draw (5,0)--(6,1.6) ;
				\draw (6,1.6)  node{$\scriptstyle \bullet$};
				
				\draw (5,0.7) node{$\scriptstyle d$};
				\draw(5.5,0)--(6.5,1.6) ;
				\draw (6.5,1.6)  node{$\scriptstyle \bullet$};
				\draw (5.5,0)  node{$\scriptstyle \bullet$};
				\draw[red] (6,0)--(7,1.6) ;
				\draw (7,1.6)  node{$\scriptstyle \bullet$};
				\draw(6,0)  node{$\scriptstyle \bullet$};
				\draw (7,1.75)  node{$\scriptstyle (1,nd+2)$};
				\draw(6,-0.2)  node{$\scriptstyle (1,d+2)$};
				
				\draw[very thick](0,0)--(8,0);
				
			\end{tikzpicture}
			\caption{The Auslander quiver of $\mathcal{C}_{A_{n}}^{d}$ when $d$ is even.}
			\label{AR-2}
		\end{figure}
		\begin{theorem}\label{main5} Let $\mathcal{C}_{A_{n}}^{d}$ be the $d$-cluster category of type $A_{n}$. Then $$K_0(\mathcal{C}_{A_{n}}^{d})\cong\left\{
			\begin{aligned}
				\mathbb{Z}/(n+1)\mathbb{Z}, &~~\text{if $d$ is even,}\\
				\mathbb{Z}, &~~ \text{if $d$ and $n$ are odd,}\\
				0, & ~~\text{if $d$ is odd and $n$ is even}.
			\end{aligned}
			\right.
			$$
		\end{theorem}
		\begin{proof}  	Set $W=d(n+1)+2$. We  emphasize that all operations on vertices in regular $W$-gon are considered modulo $W$. We divide	the proof into the following steps:
			
			$\mathbf{(Step~1)}$ For $1\leq i \leq n$, we take
			$$T_i=\left\{
			\begin{aligned}
				(\dfrac{i+1}{2}d+1,W-\dfrac{i-1}{2}d), &  &\text{if $i$ is odd;}~ \\
				(\dfrac{i+2}{2}d,W-1-\frac{i-2}{2}d), &  & \text{if $i$ is even.}~
			\end{aligned}
			\right.
			$$
			We first note that each $T_i$ is an indecomposable object in $\mathcal{C}_{A_{n}}^{d}$. Indeed, Proposition \ref{P-6-2} implies that~$T=\oplus_{i=1}^{n}T_i$ is a $(d+1)$-cluster tilting object in $\mathcal{C}_{A_{n}}^{d}$ (See Figure \ref{F-1}).
			\begin{figure}[h] 
				\begin{center}

					\tikzset{every picture/.style={line width=0.75pt}} 
					
					\begin{tikzpicture}[x=0.75pt,y=0.75pt,yscale=-1,xscale=1]
						
						\draw   (192.8,133.17) .. controls (190.69,101.61) and (214.57,74.31) .. (246.13,72.21) .. controls (277.69,70.1) and (304.99,93.98) .. (307.09,125.54) .. controls (309.2,157.1) and (285.32,184.4) .. (253.76,186.5) .. controls (222.2,188.61) and (194.91,164.73) .. (192.8,133.17) -- cycle ;
						\draw [color={rgb, 255:red, 208; green, 2; blue, 27 }  ,draw opacity=1 ]   (233.05,74.35) -- (307.09,125.54) ;
						\draw [color={rgb, 255:red, 208; green, 2; blue, 27 }  ,draw opacity=1 ]   (216.48,82.38) -- (305.48,145.38) ;
						\draw [color={rgb, 255:red, 208; green, 2; blue, 27 }  ,draw opacity=1 ]   (205,94.93) -- (296.48,161.38) ;
						\draw [color={rgb, 255:red, 208; green, 2; blue, 27 }  ,draw opacity=1 ][fill={rgb, 255:red, 208; green, 2; blue, 27 }  ,fill opacity=1 ]   (192.8,133.17) -- (262.48,185.38) ;
						
						\draw (221,64.6) node [anchor=north west][inner sep=0.75pt]  [font=\tiny]  {$W$};
						\draw (260,84.33) node [anchor=north west][inner sep=0.75pt]  [font=\tiny,color={rgb, 255:red, 240; green, 21; blue, 21 }  ,opacity=1 ]  {$T_{1}$};
						\draw (309.09,128.94) node [anchor=north west][inner sep=0.75pt]  [font=\tiny]  {$d+1$};
						\draw (258,102.33) node [anchor=north west][inner sep=0.75pt]  [font=\tiny,color={rgb, 255:red, 240; green, 21; blue, 21 }  ,opacity=1 ]  {$T_{2}$};
						\draw (307.48,148.78) node [anchor=north west][inner sep=0.75pt]  [font=\tiny]  {$2d$};
						\draw (188,73.6) node [anchor=north west][inner sep=0.75pt]  [font=\tiny]  {$W-1$};
						\draw (298.48,164.78) node [anchor=north west][inner sep=0.75pt]  [font=\tiny]  {$2d+1$};
						\draw (173,88.6) node [anchor=north west][inner sep=0.75pt]  [font=\tiny]  {$W-d$};
						\draw (256,123.33) node [anchor=north west][inner sep=0.75pt]  [font=\tiny,color={rgb, 255:red, 240; green, 21; blue, 21 }  ,opacity=1 ]  {$T_{3}$};
						\draw (221,138.33) node [anchor=north west][inner sep=0.75pt]    {$\cdots \cdots $};
						\draw (252,166.33) node [anchor=north west][inner sep=0.75pt]  [font=\tiny,color={rgb, 255:red, 240; green, 21; blue, 21 }  ,opacity=1 ]  {$T_{n}$};
					\end{tikzpicture}
					\caption{The $(d+2)$-angulation~$T=\oplus_{i=1}^{n}T_i$.}
					\label{F-1}
				\end{center}
			\end{figure}
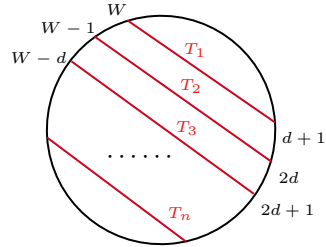

			$\mathbf{(Step~2)}$ Observe that
			$$\Sigma T_1=(W-1,d)=(W-1,W+d),$$
			$$T_2=(2d,W-1)=(W-1,W+2d).$$
			Set $X=(W-1+d,W+2d)$. Then there exists an Auslander-Reiten triangle (see Figure \ref{AR-3}) as follows:
			$$\Sigma T_1\longrightarrow T_2\stackrel{f}\longrightarrow X\longrightarrow \Sigma^{2}T_1$$		
			On the other hand, we have 
			\begin{align*}
				\Sigma^{d-1}X&=(W+2d-d+1,W-1+d-d+1)\\
				&=(d+1,W)=T_1.
			\end{align*}
			
			If $d\geq 2$, we have $\Hom(T_i,T_1)=0$ for any $i\neq 1$. Hence $0\rightarrow T_1$ is a right almost split morphism in $\mathcal{T}$. Similarly, we have  $\Hom(T_i,\Sigma^{s}T_1)=0$ for any $-d+2\leq s\leq -1$. Thus $0\rightarrow \Sigma^{s}T_1$ is a minimal right $\mathcal{T}$-approximation for	$-d+2\leq s\leq -1$. For 	$\Sigma^{1-d} T_1$, we have $\Hom(T_i,\Sigma^{1-d} T_1)\neq0$ if and only if $i=2$. Thus $f:T_2\rightarrow \Sigma^{1-d}T_1$ is a minimal right $\mathcal{T}$-approximation. It is obvious that $\Hom(T_i,\Sigma T_1)=0$. Thus $0\rightarrow \Sigma T_1$ is a minimal right $\mathcal{T}$-approximation.  By  \cite[Theorem 3.10]{Iy}, there exists an Auslander–Reiten $(d+3)$-angle starting at $T_1$ and ending at $T_1$:
			\begin{equation*}
				\begin{tikzcd}[column sep =10.5, row sep =10.5]
					&0 \ar[dr]  \ar[rr]&& T_2 \ar[dr,"f"] \ar[r] & \cdots  &  \cdots &  \cdots\ar[r]&0 \ar[dr]\ar[rr]  & & 0   \ar[dr]&\\
					T_1\ar[ur] &&\Sigma T_1 \ar[ll,rightarrow]\ar[ur]&   &\Sigma^{1-d} T_1  \ar[ll,rightarrow] \ar[r]& \cdots\ar[r] &\Sigma^{-2}T_1   \ar[ur]& &\Sigma^{-1}T_1\ar[ur]\ar[ll,rightarrow]	&   &T_1.\ar[ll,rightarrow]
				\end{tikzcd}
			\end{equation*}
			This shows that $[T_1]+[T_2]+(-1)^{d+2}[T_1]=0$ in  $K_{0}^{\rm in}(\add T)$.  
			\begin{figure}
				\centering
				\begin{tikzpicture}[scale=1.5][H]
					\draw[red](1,1.6)--(0,0);
					\draw (-0.2,-0.2)  node{$\scriptstyle (1,d+2)$};
					\draw[very thick](0,1.6)--(5,1.6);
					\draw (0,0)  node{$\scriptstyle \bullet $ };
					\draw (1,1.6)  node{$\scriptstyle \bullet$};
					\draw (0.9,1.75)  node{$\scriptstyle (1,nd+2)$};
					\draw(1.5,1.6)--(0.5,0);
					\draw (0.5,0)  node{$\scriptstyle \bullet $ };
					\draw (1.5,1.6)  node{$\scriptstyle \bullet$};
					\draw (1.5,1.75)  node{$\scriptstyle T_1 $ };
					\draw (0.7,-0.2)  node{$\scriptstyle (1+d,2d+2) $ };
					\draw(1.5,1.6)--(2.5,0);
					\draw (2.5,0)  node{$\scriptstyle \bullet$};
					\draw (2.3,-0.2)  node{$\scriptstyle (nd+1,W) $ };
					\draw(2,1.6)--(3,0);
					\draw (2,1.6)  node{$\scriptstyle \bullet$};
					\draw (3,0)  node{$\scriptstyle \bullet$};
					\draw(3,0)--(4,1.6);
					\draw(2.5,1.6)--(3.5,0);
					\draw (3.25,0.4)  node{$\scriptstyle \bullet$};
					\draw (3.45,0.4)  node{$\scriptstyle T_2 $ };
					\draw (3,-0.2)  node{$\scriptstyle \Sigma T_1 $ };
					\draw (3.5,-0.2)  node{$\scriptstyle X $ };
					\draw (3.5,0)  node{$\scriptstyle  \bullet$ };
					\draw[very thick](0,0)--(5,0);
				\end{tikzpicture}
				\caption{The Auslander-Reiten triangle $\Sigma T_1\longrightarrow T_2\stackrel{}\longrightarrow X\longrightarrow \Sigma^{2} T_1$.}
				\label{AR-3}
			\end{figure}

			$\mathbf{(Step~3)}$  Assume that $2\leq i\leq n-1$. Consider the following two cases:
			
			$\mathbf{(Case~1})$ $i$ is even. Observe that
			$$T_{i-1}=(\dfrac{i}{2}d+1,W-\dfrac{i-2}{2}d),$$
			$$T_{i+1}=(\dfrac{i+2}{2}d+1,W-\dfrac{i}{2}d),$$
			$$\Sigma^{-1} T_i=(W-1-\dfrac{i-2}{2}d+1,\dfrac{i+2}{2}d+1)=(W-\dfrac{i-2}{2}d,\dfrac{i+2}{2}d+1),$$
			$$\Sigma^{d-1} T_i=(W-1-\dfrac{i-2}{2}d-d+1,\dfrac{i+2}{2}d-d+1)=(W-\dfrac{i}{2}d,\dfrac{i}{2}d+1).$$
			Then there exists an Auslander-Reiten triangle (see Figure \ref{AR-4}) as follows:
			$$\Sigma^{d-1} T_i\longrightarrow T_{i-1}\oplus T_{i+1}\longrightarrow \Sigma^{-1}T_i\longrightarrow\Sigma^{d} T_i.$$
			One can check that there exists  an Auslander-Reiten  starting at $T_i$ and ending at $T_i$:
			\begin{equation*}
				\begin{tikzcd}[column sep =10.5, row sep =10.5]
					&0 \ar[dr]  \ar[rr]&& 0 \ar[dr] \ar[r] & \cdots  &  \cdots &  \cdots\ar[r]&T_{i-1}\oplus T_{i+1}  \ar[dr]\ar[rr]  & & 0 \ar[dr]&\\
					T_i\ar[ur] &&\Sigma T_i \ar[ll,rightarrow]\ar[ur]&   &\Sigma^{2} T_i  \ar[ll,rightarrow] \ar[r]& \cdots\ar[r] &\Sigma^{d-1}T_i   \ar[ur]& &\Sigma^{-1}T_i\ar[ur]\ar[ll,rightarrow]	&   &T_i.\ar[ll,rightarrow]
				\end{tikzcd}
			\end{equation*}
			This shows that $[T_i]+(-1)^{d}[T_{i-1}\oplus T_{i+1}]+(-1)^{d+2}[T_i]=0$ in  $K_{0}^{\rm in}(\add T)$.

			\begin{figure}
				\centering
				\begin{tikzpicture}[scale=1.5][H]
					\draw[very thick](0,1.6)--(4,1.6);
					\draw (1.5,1.6)  node{$\scriptstyle \bullet$};
					\draw (1.5,1.75)  node{$\scriptstyle T_1=(1+d,W) $ };
					
					\draw[dashed](1.5,1.6)--(1.5,0);
					\draw(0.8,1.6)--(1.8,0);
					\draw(1.1,1.6)--(2.1,0);
					\draw (1.5,0.97)  node{$\scriptstyle \bullet$};
					\draw (1.55,1.2)  node{$\scriptstyle T_{i-1}$ };
					\draw (1.5,0.47)  node{$\scriptstyle \bullet$};
					\draw (1.6,0.22)  node{$\scriptstyle T_{i+1}$ };
					\draw(1.9,1.6)--(0.9,0);
					\draw(2.2,1.6)--(1.2,0);
					\draw (1.35,0.7)  node{$\scriptstyle \bullet$};
					\draw (0.95,0.7)  node{$\scriptstyle \Sigma^{d-1}T_i$ };
					\draw (1.65,0.7)  node{$\scriptstyle \bullet$};
					\draw (2.05,0.7)  node{$\scriptstyle \Sigma^{-1}T_i$ };
					\draw[dashed](2.8,1.6)--(2.8,0);
					\draw (2.8,0.2)  node{$\scriptstyle \bullet$};
					\draw (3,0.2)	node{$\scriptstyle T_2$ };
					\draw (2.8,0.9)  node{$\scriptstyle \bullet$};
					\draw[very thick](0,0)--(4,0);
					\draw (3,0.9)	node{$\scriptstyle T_i$ };
					~~~
				\end{tikzpicture}
				\caption{The Auslander-Reiten triangle $\Sigma^{d-1} T_i\longrightarrow T_{i-1}\oplus T_{i+1}\longrightarrow \Sigma^{-1}T_i\longrightarrow\Sigma^{d} T_i$.}
				\label{AR-4}
			\end{figure}
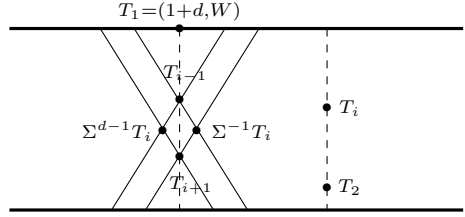
			\begin{figure}
				\centering
				\begin{tikzpicture}[scale=1.5][H]
					\draw[very thick](0,1.6)--(4,1.6);
					\draw (1.2,1.6)  node{$\scriptstyle \bullet$};
					\draw (1.2,0.8)  node{$\scriptstyle \bullet$};
					\draw (1,0.8)  node{$\scriptstyle T_{i}$ };
					\draw[dashed](1.2,1.6)--(1.2,0);
					\draw(2.36,1.6)--(3.36,0);
					\draw(2.1,1.6)--(3.1,0);
					\draw(3.23,1.6)--(2.23,0);
					\draw(3.48,1.6)--(2.48,0);
					\draw[dashed](2.8,1.6)--(2.8,0);
					\draw (2.8,0.5)  node{$\scriptstyle \bullet$};
					\draw (2.83,0.2)  node{$\scriptstyle T_{i+1}$ };
					\draw (2.8,0.9)  node{$\scriptstyle \bullet$};
					\draw[very thick](0,0)--(4,0);
					\draw (2.8,1.2)	node{$\scriptstyle T_{i-1}$ };
					\draw (1.2,1.75)  node{$\scriptstyle T_1=(1+d,W) $ };
					\draw (2.92,0.7)  node{$\scriptstyle \bullet$};
					\draw (3.4,0.7)  node{$\scriptstyle \Sigma^{1-d}T_i$ };
					\draw (2.67,0.7)  node{$\scriptstyle \bullet$};
					\draw (2.37,0.7)  node{$\scriptstyle \Sigma T_i$ };
				\end{tikzpicture}
				\caption{The Auslander-Reiten triangle $\Sigma T_i\longrightarrow T_{i-1}\oplus T_{i+1}\longrightarrow \Sigma^{1-d} T_i\longrightarrow \Sigma^{2}T_i$.}
				\label{AR-5}
			\end{figure}
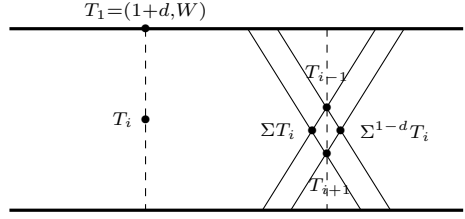

			$\mathbf{(Case~2})$ $i$ is odd. This case is similar to that of the first one. Observe that
			$$T_{i-1}=(\dfrac{i+1}{2}d,W-1-\dfrac{i-3}{2}d),$$
			$$T_{i+1}=(\dfrac{i+3}{2}d,W-1-\dfrac{i-1}{2}d),$$
			$$\Sigma T_i=(W-1-\dfrac{i-1}{2}d,\dfrac{i+1}{2}d),$$
			$$\Sigma^{1-d} T_i=(W-\dfrac{i-1}{2}d+d-1,\dfrac{i+1}{2}d+1+d-1)=(W-1+\dfrac{3-i}{2}d,\dfrac{i+3}{2}d).$$
			Then there exists an Auslander-Reiten triangle (see Figure \ref{AR-5}) as follows:
			$$\Sigma T_i\longrightarrow T_{i-1}\oplus T_{i+1}\longrightarrow \Sigma^{1-d} T_i\longrightarrow \Sigma^{2}T_i.$$
			One can check that there exists an  Auslander-Reiten  $(d+3)$-angle starting at $T_i$ and ending at $T_i$:
			\begin{equation*}
				\begin{tikzcd}[column sep =10.5, row sep =10.5]
					&0 \ar[dr]  \ar[rr]&& T_{i-1}\oplus T_{i+1} \ar[dr] \ar[r] & \cdots  &  \cdots &  \cdots\ar[r]&0 \ar[dr]\ar[rr]  & & 0 \ar[dr]&\\
					T_i\ar[ur] &&\Sigma T_i \ar[ll,rightarrow]\ar[ur]&   &\Sigma^{1-d} T_i  \ar[ll,rightarrow] \ar[r]& \cdots\ar[r] &\Sigma^{-2}T_i   \ar[ur]& &\Sigma^{-1}T_i\ar[ur]\ar[ll,rightarrow]	&   &T_i.\ar[ll,rightarrow]
				\end{tikzcd}
			\end{equation*}
			This shows that $[T_i]+[T_{i-1}\oplus T_{i+1}]+(-1)^{d+2}[T_i]=0$ in  $K_{0}^{\rm in}(\add T)$.
			
			$\mathbf{(Step~4)}$  It remains to consider $T_n$. We consider two cases:
			
			$\mathbf{(Case~1})$ $n$ is even. Then
			$$\Sigma^{-1} T_n=(W-\dfrac{n-2}{2}d,\dfrac{n+2}{2}d+1).$$
			$$T_{n-1}=(\dfrac{n}{2}d+1,W-\dfrac{n-2}{2}d),$$
			$$\Sigma^{d-1} T_n=(W-\dfrac{n}{2}d,\dfrac{n}{2}d+1).$$
			It is shown in Figure \ref{AR-6} that there is  an Auslander-Reiten triangle 
			$$\Sigma^{d-1} T_n\longrightarrow T_{n-1}\longrightarrow  \Sigma^{-1}T_n \longrightarrow \Sigma^{d} T_n.$$
			One can check that there exists an  Auslander-Reiten  $(d+3)$-angle starting at $T_n$ and ending at $T_n$:
			\begin{equation*}
				\begin{tikzcd}[column sep =10.5, row sep =10.5]
					&0 \ar[dr]  \ar[rr]&& 0 \ar[dr] \ar[r] & \cdots  &  \cdots &  \cdots\ar[r]&T_{n-1} \ar[dr]\ar[rr]  & & 0 \ar[dr]&\\
					T_n\ar[ur] &&\Sigma T_n \ar[ll,rightarrow]\ar[ur]&   &\Sigma^{2} T_n  \ar[ll,rightarrow] \ar[r]& \cdots\ar[r] &\Sigma^{d-1}T_n  \ar[ur]& &\Sigma^{-1}T_n\ar[ur]\ar[ll,rightarrow]	&   &T_n.\ar[ll,rightarrow]
				\end{tikzcd}
			\end{equation*}
			Thus $[T_n]+(-1)^{d}[T_{n-1}]+(-1)^{d+2}[T_{n}]=0$ in  $K_{0}^{\rm sp}(\add T)/\im\Psi$.
			
			$\mathbf{(Case~2})$ $n$ is odd.   Observe that
			$$\Sigma T_n=(W-1-\dfrac{n-1}{2}d,\dfrac{n+1}{2}d).$$
			$$T_{n-1}=(\dfrac{n+1}{2}d,W-1-\dfrac{n-3}{2}d),$$
			$$\Sigma^{1-d} T_n=(W-1-\dfrac{n-3}{2}d,\dfrac{n+3}{2}d).$$
			It is shown in Figure \ref{AR-7} that  there exists an Auslander-Reiten triangle 
			$$\Sigma T_n\longrightarrow T_{n-1}\longrightarrow \Sigma^{1-d} T_n\longrightarrow \Sigma^{2}T_n.$$
			One can check that there exists an  Auslander-Reiten  $(d+3)$-angle starting at $T_n$ and ending at $T_n$:
			\begin{equation*}
				\begin{tikzcd}[column sep =10.5, row sep =10.5]
					&0 \ar[dr]  \ar[rr]&& T_{n-1} \ar[dr] \ar[r] & \cdots  &  \cdots &  \cdots\ar[r]&0 \ar[dr]\ar[rr]  & & 0 \ar[dr]&\\
					T_n\ar[ur] &&\Sigma T_n \ar[ll,rightarrow]\ar[ur]&   &\Sigma^{1-d} T_n  \ar[ll,rightarrow] \ar[r]& \cdots\ar[r] &\Sigma^{-2}T_n  \ar[ur]& &\Sigma^{-1}T_n\ar[ur]\ar[ll,rightarrow]	&   &T_n.\ar[ll,rightarrow]
				\end{tikzcd}
			\end{equation*}
			Thus $[T_n]+[T_{n-1}]+(-1)^{d+2}[T_n]=0$ in  $K_{0}^{\rm in}(\add T)$.
			\begin{figure}
				\centering
				\begin{tikzpicture}[scale=1.5][H]
					\draw[very thick](0,1.6)--(4.2,1.6);
					\draw (1.26,1.6)  node{$\scriptstyle \bullet$};
					\draw (2.8,1.6)  node{$\scriptstyle \bullet$};
					\draw(2.8,1.6)--(1.8,0);	
					\draw(0.8,1.6)--(1.8,0);
					\draw(0.5,1.6)--(1.5,0);
					\draw(1.02,0)--(2.02,1.6);
					\draw (2.8,1.75) node{$\scriptstyle T_n$ };
					
					\draw[dashed](1.26,1.6)--(1.26,0);
					\draw[dashed](2.8,1.6)--(2.8,0);
					\draw[very thick](0,0)--(4.2,0);
					\draw (1.25,1.75)  node{$\scriptstyle T_1 $ };
					\draw (1.27,0.38)  node{$\scriptstyle \bullet$};
					\draw (1.5,0)  node{$\scriptstyle \bullet$};
					\draw (1.6,-0.2)  node{$\scriptstyle \Sigma^{-1} T_n $ };
					\draw (1,0.38)  node{$\scriptstyle  T_{n-1} $ };
					\draw (1.04,0)  node{$\scriptstyle   \bullet$};
					\draw (0.9,-0.2)  node{$\scriptstyle   	\Sigma^{d-1} T_n$};
					
				\end{tikzpicture}
				\caption{The Auslander-Reiten triangle 
					$\Sigma^{d-1} T_n\longrightarrow T_{n-1}\longrightarrow  \Sigma^{-1}T_n \longrightarrow \Sigma^{d} T_n$.}
				\label{AR-6}
			\end{figure}
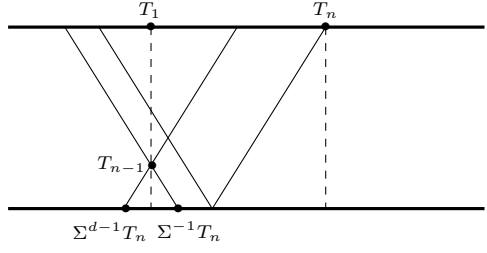

			\begin{figure}
				\centering
				\begin{tikzpicture}[scale=1.5][H]
					\draw[very thick](0,1.6)--(4.2,1.6);
					\draw (1.26,1.6)  node{$\scriptstyle \bullet$};
					\draw (2.51,1.6)  node{$\scriptstyle \bullet$};
					\draw (2.51,1.75)  node{$\scriptstyle \Sigma T_n$};
					\draw(3.26,0)--(2.26,1.6);	
					\draw(2.5,1.6)--(3.5,0);
					\draw(3.1,1.6)--(2.1,0);
					\draw(1.26,0)--(2.26,1.6);
					\draw (2.8,1.13) node{$\scriptstyle \bullet$ };
					\draw (3.1,1.1) node{$\scriptstyle T_{n-1}$ };
					
					\draw[dashed](1.26,1.6)--(1.26,0);
					\draw[dashed](2.8,1.6)--(2.8,0);
					\draw[very thick](0,0)--(4.2,0);
					\draw (1.25,1.75)  node{$\scriptstyle T_1 $ };
					\draw (1.26,0)  node{$\scriptstyle   \bullet$};
					\draw (1.26,-0.2)  node{$\scriptstyle   T_n$};
					\draw (3.2,1.75)  node{$\scriptstyle   \Sigma^{1-d} T_n $};	
					\draw (3.1,1.6)  node{$\scriptstyle   \bullet $};	
				\end{tikzpicture}
				\caption{The Auslander-Reiten triangle 
					$\Sigma T_n\longrightarrow T_{n-1}\longrightarrow \Sigma^{1-d} T_n\longrightarrow \Sigma^{2}T_n.$}
				\label{AR-7}
			\end{figure}
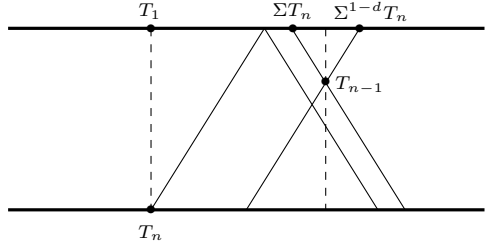

			$\mathbf{(Step~5)}$ By (Step 1)--(Step 4), we conclude that
			
			$\bullet$ If $d$ is even, then 
			$$\left\{
			\begin{aligned}
				[T_2]&=-2[T_1],   \\
				[T_3]&=-2[T_2]-[T_1]=3[T_1], \\
				[T_4]&=-2[T_3]-[T_2]=-4[T_1], \\
				&\cdots\cdots\\
				[T_{n-1}]&=-2[T_{n-2}]-[T_{n-3}]=(-1)^{n}(n-1)[T_1], \\
				[T_{n}]&=-2[T_{n-1}]-[T_{n-2}]=(-1)^{n+1}n[T_1], \\
				[T_{n-1}]&=-2[T_n]. \\
			\end{aligned}
			\right.
			$$
			The last equality above implies that
			$$(-1)^{n}(n-1)[T_1]=-2(-1)^{n+1}n[T_1]$$
			and hence $(n+1)[T_1]=0$ in $K_{0}^{\rm in}(\add T)$.
			
			$\bullet$ If $d$ is odd, then $[T_2]=0$ and
			$$[T_{i-1}]+[T_{i+1}]=0$$
			for any $2\leq i\leq n-1$. If $n$ is even, then $[T_{n-1}]=0$ and hence 
			$$[T_1]=[T_2]=\cdots=[T_{n}]=0.$$
			If $n$ is odd, then
			$$[T_2]=[T_4]=[T_{n-1}]=0,$$
			$$[T_1]=-[T_3]=\dots=(-1)^{n+1}[T_{n}].$$
			By using Theorem \ref{main3}(3) and Proposition \ref{P-6.1}, we have
			$$K_0(\mathcal{C}_{A_{n}}^{d})\cong\left\{
			\begin{aligned}
				\mathbb{Z}/(n+1)\mathbb{Z}, &~~\text{if $d$ is even,}\\
				\mathbb{Z}, &~~ \text{if $d$ and $n$ are odd,}\\
				0, & ~~\text{if $d$ is odd and $n$ is even}.
			\end{aligned}
			\right.
			$$
		\end{proof}
		
		\begin{remark}
			We note that the case of $d$ is odd was also calculated in \cite[Proposition 7.21]{FE}.
		\end{remark}	
		
		We finish this section with an example illustrating  Theorem \ref{main5}.
		\begin{example} Let  $\mathcal{C}_{A_3}^{2}$ be the $2$-cluster  category of type $A_3$.	The Auslander-Reiten quiver is shown in Example \ref{E-1}(3). Take $T_1=(3,10),~T_2=(9,4)$ and $T_3=(5,8)$. Then $T=\oplus_{i=1}^{3}T_i$ is a $3$-cluster tilting object in $\mathcal{C}_{A_3}^{2}$. By Theorem \ref{main5}, we infer that
			$$[T_2]=-2[T_1],~[T_3]=3[T_1],~\text{and}~4[T_1]=0$$in $K_{0}^{\rm in}(\add T)$ and hence $K_0(\mathcal{C}_{A_{3}}^{2})\cong \mathbb{Z}/4\mathbb{Z}$.
		\end{example}	
		
	\end{document}